\documentclass[final]{siamltex}

\usepackage{epsfig,appendix,makecell}
\usepackage[english]{babel}
\usepackage{amssymb,amsmath,amsfonts,mathbbol}
\usepackage{mathrsfs,empheq}
\usepackage{mdwlist,fontawesome}
\usepackage{enumerate,url,enumitem,multirow,multicol}
\usepackage[usenames,dvipsnames,svgnames]{xcolor}

\usepackage{algorithm,extarrows}
\usepackage{arydshln}
\usepackage{algcompatible}
\algnewcommand\algorithmicreturn{\textbf{return}}
\algnewcommand\RETURN{\State \algorithmicreturn}%
\usepackage[outercaption]{sidecap}

\usepackage{tikz}
\usepackage{bm}

\usepackage{pgfplots}
\pgfplotsset{
  tick label style={font=\footnotesize},
  label style={font=\footnotesize},
  legend style={font=\footnotesize}
}

\usepackage{booktabs}
\usepackage{caption,subcaption}

\usepackage{todonotes,bbding,stackengine}

\usepackage[normalem]{ulem} 

\usepackage[colorlinks=true, allcolors=blue]{hyperref} 
\usepackage{cleveref}  

\newcommand{\xtilde}{\tilde{x}}

\newcommand{\Fix}{\mathrm{Fix}}

\newcommand{\R}{\mathbb{R}}
\newcommand{\N}{\mathbb{N}}
\newcommand{\cI}{\mathcal{I}}
\newcommand{\cR}{\mathcal{R}}
\newcommand{\cS}{\mathcal{S}}

\newcommand{\cN}{\mathcal{N}}

\newcommand{\cT}{\mathcal{T}}
\newcommand{\cA}{\mathcal{A}}

\newcommand{\cP}{\mathcal{P}}
\newcommand{\cJ}{\mathcal{J}}
\newcommand{\cL}{\mathcal{L}}
\newcommand{\cH}{\mathcal{H}}

\newcommand{\cQ}{\mathcal{Q}}

\newcommand{\cK}{\mathcal{K}}

\newcommand{\ztilde}{\tilde{z}}

\newcommand{\utilde}{\tilde{u}}
\newcommand{\stilde}{\tilde{s}}

\newcommand{\gra}{\mathrm{gra}}
\newcommand{\zer}{\mathrm{zer}}

\newcommand{\be}{\begin{equation}}
\newcommand{\ee}{\end{equation}}

\newcommand{\weak}{\rightharpoonup}

\newcommand{\mcup}{\mathbin{\raisebox{-0.20ex} {\scalebox{1.3}{$\cup$}}}}

\DeclareMathOperator{\prox}{prox}       
\newtheorem{assumption}{Assumption}
\newtheorem{remark}{Remark}
\newtheorem{fact}{Fact}
\usepackage[normalem]{ulem} 

\usepackage{enumitem}
\setlist[itemize]{leftmargin=1em} 

\newif\ifcompilePGFfigs
\compilePGFfigstrue

\newcommand{\tabincell}[2]{\begin{tabular}{@{}#1@{}}#2\end{tabular}}

\graphicspath{{figs/}}

\title{A frugal primal-dual splitting with minimal lifting over arbitrary rooted trees
}

\author{Feng Xue\thanks{National key laboratory, Beijing,  China (\tt{fxue@link.cuhk.edu.hk}).}
\and Hui Zhang\textsuperscript{\Envelope}\thanks{Corresponding author.  College of Science, National University of Defense Technology, Changsha, Hunan 410073, China (\tt{h.zhang1984@163.com}).} }

\date{\today}

\begin{document}

\maketitle

\begin{abstract}
We develop a frugal primal-dual splitting with minimal lifting for solving structured monotone inclusions, involving cocoercive operators, linear compositions and parallel sums. This is established by defining hierarchical nodes and edges between them over a tree-structured graph, with arbitrary assignments of dual variables and cocoercive elements to primal nodes. This arbitrariness allows a great flexibility in terms of  level-synchronous distributed computing, such as centralized or decentralized. 
The particular instances naturally extend the Douglas--Rachford splitting on various graphs, recover the parallel Chambolle--Pock, and solve a class of structured convex minimization problems. For the pure resolvent convex minimization subclass, we establish an $O(1/k)$ ergodic rate for a restricted primal-dual gap over bounded test sets.   
Furthermore, we introduce a reformulation technique in a product Hilbert space to facilitate the convergence analysis, specifically to derive the $o(1/k)$-rate of asymptotic regularity.  

\end{abstract}




\section{Introduction}
\subsection{Frugal splitting with minimal lifting}
In the realm of mathematical optimization and related fields, a wide variety of problems can be formulated as structured monotone inclusions. The best-known solvers include the proximal point method \cite{rtr_1976_2}, and splitting-type methods such as the Douglas–Rachford splitting (DRS) algorithm \cite{lions}, the forward-backward splitting algorithm \cite{plc}, and the Davis–Yin three-operator splitting algorithm \cite{ywt_2017}. Motivated by recent advances in operator splitting techniques \cite{ryu_lift}, there has been a growing focus on developing frugal resolvent splitting methods, which utilize only vector additions, scalar multiplications, and exactly one resolvent evaluation of each monotone operator per iteration \cite{yura_lift,fjaa_pds,bredies_lift,tam_2026}. Building on these developments, the most general form of monotone inclusions solvable by these methods is given as:
\be \label{p_simple}
\textrm{find } u^\star \in \cH, \  \textrm{such that} \ 
0 \in  \Bigg( \sum_{i=1}^{n} A_i + \sum_{j=1}^m \big( L_j^*\circ  
 B_j  \circ L_j \big)  +  \sum_{l=1}^{p} C_l  \Bigg)  
(u^\star) ,
\ee
particularly for the case of $n\ge 2$. Here,
 $A_i: \cH\mapsto 2 ^\cH$ and $B_j: \cK_j \mapsto 2^{\cK_j}$ are set-valued maximally monotone, defined on real Hilbert spaces $\cH$ and $\cK_j$, respectively.  $C_l$ for $l=1,\ldots,p$ are $\beta_l$-cocoercive, $L_j : \cH \mapsto \cK_j$ are bounded  and linear. 
 
To mitigate memory overhead in splitting-based algorithms, the notion of lifting was introduced in the work \cite{ryu_lift}, which established the DRS algorithm as the unique non-lifted unconditionally convergent resolvent splitting for the case $n=2$, $m=0$ and $p=0$ \cite[Corollary 1]{ryu_lift},
and  proposed an algorithm with minimal 2-fold lifting for $n=3$  \cite[Theorems 3 and 4]{ryu_lift}. Later on, the authors of \cite{yura_lift} extended this result to $n\ge 2$,  showing that the minimal lifting is $(n-1)$-fold \cite[Theorem 3.3]{yura_lift}.  Several distributed variants of DRS have been developed in \cite{condat_tour,bredies_ppa,campoy_product} for general $n\ge 2$. Subsequent works further extended these results to the case involving cocoercive operators \cite{morin_frugal,fjaa_ring,full}.
In particular, \cite{fjaa_pds} appears to be the first attempt to deal with the involvement of the linear compositions (i.e., $m >0$). Specifically, \cite[Theorem 5]{fjaa_pds} proved that the minimal lifting is $(n-1, m)$-fold for $n\ge 2$, $m\in\N$ and $p=0$.  This was later slightly improved by \cite{tyc_pds} to $p\in \N$. 
Most recently, the graph understanding has attracted increasing attention, particularly for the case of $n\ge 2$, $m=0$ and $p=0$. A number of algorithms have been developed or interpreted on bilevel graph \cite{bredies_lift}, ring network \cite{fjaa_ring} and other graph structures \cite{fjaa_graph,tam_2026}. 
Table \ref{table_summary} summarizes the key features of typical frugal splitting methods with minimal lifting, particularly for the case of $n\ge 2$.
\begin{table}[h!]
  \centering
   \caption{Summary of existing splitting methods.}\vspace{-1em}
   \resizebox{1.0\columnwidth}{!} {
 \begin{tabular}{|l||l|l|l|l|l|} 
    \hline
    scheme & \makecell[l]{ involving \\ cocoercive }
    & \makecell[l]{ involving \\ linear \\ compositions} &
    \makecell[l]{ involving \\ parallel sums} &
     \makecell[l]{ distributed \\ computing}    & \makecell[l]{ recovery of \\ Chambolle--Pock \\ \cite[Algorithm 1]{cp_2011} }  \\
    \hline\hline
    \makecell[l]{ DRS \cite{lions} \\ 
    Ryu    \cite[Theorem 4]{ryu_lift} \\
    Malitsky--Tam  \cite[Algorithm 1]{yura_lift} \\
    Extended Ryu  \cite[Sect. 3.1]{bredies_lift} \\
    \cite[Example 3.4]{tam_optl}  } 
    &   \faTimes   & \faTimes    & \faTimes  & sequential & \faTimes     \\
\hline
 \makecell[l]{ Davis--Yin  \cite[Algorithm 1]{ywt_2017} \\
  forward-DRS   \cite[Eq. (3.15)]{bredies_ppa}  \\ 
 distributed FB  \cite[Algorithm 1]{fjaa_ring} }     
   & \faCheck    & \faTimes    & \faTimes  & sequential & \faTimes \\
 \hline
  \makecell[l]{  \cite[Theorem 8.1]{morin_frugal} \\
   forward-DRS  \cite[Eq. (3.14)]{bredies_ppa}  }   & \faCheck    & \faTimes    & \faTimes  & parallel & \faTimes  \\
 \hline
  \makecell[l]{ Campoy    \cite[Theorem 5.1]{campoy_product}  }   & \faTimes    & \faTimes    & \faTimes  & parallel & \faTimes  \\ 
 \hline
  \tabincell{l}{   \cite[Algorithms 3 and 4]{bredies_lift}  
  \\ \cite[Algorithm 2]{bot_2024} } & \faTimes    & \faTimes    & \faTimes  & flexible & \faTimes  \\
 \hline
  \tabincell{l}{   \cite[Algorithm 1]{tam_2026}  }   & \faTimes    & \faCheck    & \faTimes  & sequential & \faTimes   \\
 \hline
   \makecell[l]{   \cite[Algorithm 1]{fjaa_pds}  }   & \faTimes    & \faCheck    & \faTimes  & sequential & \faTimes  \\
 \hline
   \makecell[l]{ \cite[Algorithm 1]{tyc_pds}  }   & \faCheck    & \faCheck    & \faTimes  & sequential & \faTimes  \\
\hline
  \tabincell{l}{ ours: Eq. \eqref{proto_T_relax}  }     & \faCheck    & \faCheck    & \faCheck  & flexible & \faCheck \\
\hline
    \end{tabular}
   }
    \vskip 0.5em
\label{table_summary}
\vspace*{-0.3cm}
\end{table}

In this work, we extend \eqref{p_simple} to the following problem involving parallel sums\footnote{The numbering of $i$ adopted here is from $0$ to $n-1$. This is for convenience of later developments, to distinguish the index of $0$ from other numbers. The index of $0$ will be referred to as a root of a tree, see Figure \ref{fig_whole}.}
\be \label{p}
\textrm{find } u^\star \in \cH, \  \textrm{such that} \ 
a \in  \Bigg( \sum_{i=0}^{n-1} A_i + \sum_{j=1}^m (L_j^*\circ  
(B_j \square D_j) \circ (L_j \cdot - b_j ) +  \sum_{l=1}^{p} C_l  \Bigg)  (u^\star) ,
\ee
where $B_j \square D_j$ denotes the parallel sum of $B_j$ and $D_j$, which is defined as $B_j \square D_j = (B_j^{-1} + D_j^{-1})^{-1}$. Here, $D_j: \cK_j \mapsto 2^{\cK_j}$ is $\nu_j$-strongly maximally monotone.   
The form \eqref{p} is highly general and encompasses  \eqref{p_simple} as its special instance  by setting $a=0$, $b_j=0$, and 
$D_j = \cN_{\{0\}} $---a normal cone to the singleton $\{0\}$, which yields $B_j \square D_j = B_j$ \cite[Example 1.3]{vu_2013}. 
Moreover, \eqref{p} also covers several early works addressing parallel-sum terms \cite{plc_2012,plc_2013,plc_2018,vu_2013,bot_jmiv_2014}. Throughout this paper, we always assume that {\it the solution set of \eqref{p} is non-empty}, unless otherwise specified.

\subsection{Contributions}
The primary purpose of this work is {\it not} to provide a full characterization of the frugal splitting operator (as comprehensively addressed in \cite{full,morin_frugal,tam_2026}). Instead, we 
propose a novel splitting method (formulated as \eqref{proto_T_relax}) tailored for solving \eqref{p}, which highlights the following key features (also see the last row of Table \ref{table_summary}):
\begin{itemize}
\item To the best of our knowledge, this seems the {\it first} splitting algorithm for solving \eqref{p} with any finite numbers of $n \ge 2$ and $m,p\in\N$, involving a mixture of cocoercive operators, linear compositions and parallel sums, via the equipped fixed-point encoding (cf. Propsition \ref{p_encoding}). 
\item This is {\it frugal},   {\it unconditionally convergent} and achieves the {\it minimal lifting of $(n-1,m)$-fold} (cf. Theorem \ref{t_encoding} and Theorem \ref{t_zs}).
\item This is directly built on a tree with {\it arbitrary} arrangements of primal nodes, assignments of dual variables and loadings of cocoercive operators (cf. Fig. \ref{fig_whole}). The arbitrariness of the tree topology allows a great  flexibility for {\it parameter selection} and {\it level-synchronous} distributed computing. 
\end{itemize}

\vskip.2cm
We extend existing works in the following practical and theoretical aspects:

(i) {\it Wide applicability}: The general problem \eqref{p} cannot be solved by existing splitting methods. Early algorithms \cite{plc_2018,vu_2013,bot_jmiv_2014} addressed the case of $n=1$ only, while most of the recent frugal algorithms such as \cite{yura_lift,morin_frugal,tam_2026} cannot deal with linear compositions. Though  \cite{fjaa_pds,tyc_pds} could solve the linear compositions, they do not align with the Chambolle--Pock style \cite[Algorithm 1]{cp_2011}, since they relied on $B_j$ rather than $B_j^{-1}$. It would be difficult to extend  \cite{fjaa_pds,tyc_pds} to the parallel-sum type. 
This work addresses these issues to solve the general \eqref{p}, by leveraging the resolvent of $B_j^{-1}$ {\it \`{a} la} Chambolle--Pock \cite[Algorithm 1]{cp_2011}.

(ii) {\it Great extension}: The proposed method (cf. \eqref{proto} or \eqref{proto_T_relax}) is essentially a {\it consensus-ADMM}, incorporating  primal-dual splitting  {\it \`{a} la} Chambolle--Pock \cite{cp_2011}, and forward-backward splitting  {\it \`{a} la} the classical proximal iteration \cite{plc} as basic ingredients. Therefore, it can be seen as an extension or mixture of those above standard solvers.

(iii) {\it Simple graph understanding}: Most existing works parameterize fixed-point frugal splitting operators using lower triangular structures with minimal lifting \cite{full,morin_frugal,tam_2026} or specific graph topologies \cite{bredies_lift,fjaa_graph}. Notably, approaches such as the bilevel graph \cite{bredies_lift} and additional subgraph \cite{fjaa_graph} are designed to distinguish updates of primal nodes (state graph) from those of interaction variables (base graph) and cocoercive operators (subgraph).
Our approach is fundamentally distinct from the aforementioned  methods, in that we directly construct a quite simple  {\it single layered} tree according to a certain protocol, as illustrated in Fig. \ref{fig_whole}, which seems the first graph to tackle the general problem \eqref{p}.
This unifies the treatments of primal nodes, interaction variables, dual and cocoercive elements. More importantly, inspired by consensus-ADMM, we give a straightforward interpretation for the metric weights of the tree edges via augmented Lagrangian (cf. Sect. \ref{sec_dynamical}).

(iv) {\it Great flexibility}: Most existing works adopt a fixed strategy to handle cocoercive operators and linear compositions \cite{fjaa_pds,morin_frugal,bredies_ppa}. In contrast, our proposed arbitrary tree structure, with flexible assignments of dual variables and loadings of cocoercive operators, allows for significant versatility in various distributed implementations

\subsection{Notations and assumptions}  \label{sec_notation}
We use standard notations and concepts from convex analysis and variational analysis, which, unless otherwise specified, can all be found in the classical and recent monographs \cite{rtr_book,rtr_book_2,plc_book}. The notational conventions regarding the tree graph are detailed in Remarks \ref{rmk_notation} and \ref{rmk_notation2}. 

A few more words about our  notations are in order. The set of integers $\{1,\ldots,n\}$ starting from 1  (excluding zero) is symbolized as $[n]$ for short.  
A resolvent of $A$ parametrized by a self-adjoint and positive definite metric $M$ is given as $J_A^M := J_{M^{-1}A} \circ   M^{-1}$, where $J_{M^{-1}A}$ follows a standard definition of resolvent \cite[Definition 23.1]{plc_book}.
For a linear, bounded, self-adjoint and positive semi-definite (i.e., possibly degenerate) $\cQ$, $\cQ$-semi inner product and seminorm are defined as 
$\langle \cdot |  \cdot \rangle_\cQ = \langle \cdot | \cQ \cdot \rangle $ and
$\|\cdot\|_\cQ = \sqrt{\langle \cdot |  \cdot \rangle_\cQ} $.
Throughout this paper, we normally use  inclusion form to express the algorithms instead of the commonly used resolvents in literature, since the inclusion form is (i) intrinsically closer to augmented Lagrangian formulation and the graph understanding; (ii) easier to be reformulated in a product Hilbert space for convergence analysis.

\section{Algorithmic prototype}
This section develops a frugal splitting prototype based on a tree graph. 

\subsection{Splitting strategy}
To begin with, we reformulate \eqref{p} as the following primal-dual inclusions \eqref{pd} and splitting form \eqref{pds}. They can also be treated as the optimality conditions of the problem \eqref{p}.
\begin{lemma}  \label{l_pd}
If $u^\star \in \cH$ is a solution to the primal inclusion \eqref{p}, then

{\rm (i)} $\exists (s_j^\star)_{j=1,\ldots,m} \in \prod_{j=1}^m \cK_j$, such that 
\be \label{pd}
\left\{   \begin{aligned}
0 & \in  \sum_{i=0}^{n-1} A_i u^\star + \sum_{j=1}^m L_j^* s_j^\star + \sum_{l=1}^p C_l u^\star - a; \\
0 & \in (B_j^{-1}  + D_j^{-1} ) s_j^\star - L_j u^\star + b_j, \quad \forall j\in [m].
\end{aligned}  \right.
\ee

{\rm (ii)}  $\exists (u_i^\star)_{i=1,\ldots,n-1} \in \cH^{n-1}$,
 $\exists (w_i^\star)_{i=1,\ldots,n-1} \in \cH^{n-1}$,
  $\exists (s_j^\star)_{j=1,\ldots,m} \in \prod_{j=1}^m \cK_j$,  such that  
\be \label{pds}
\left\{   \begin{aligned}
0 & \in A_0 u_0^\star - \sum_{i=1}^{n-1} w_i^\star 
+  \sum_{l \in \cP_0}  C_l u_0^\star 
+ \sum_{j \in \cJ_0}  L_j^* s_j^\star  - a; \\
0 & \in A_i u_i^\star + w_i^\star + \sum_{l \in \cP_i}  C_l u_0^\star 
+ \sum_{j \in \cJ_i}  L_j^* s_j^\star , 
\quad \forall i\in [n-1]; \\
0 & \in B_j^{-1} s_j^\star + D_j^{-1} s_j^\star - L_j u_0^\star + b_j, \quad \forall j\in [m]; \\
0 & = u_0^\star - u_i^\star, \quad \forall i\in [n-1], 
\end{aligned}  \right.
\ee
where $\{\cP_i\}_{i=0}^{n-1}$  and $\{\cJ_i\}_{i=0}^{n-1}$ form arbitrary partitions of $[p]$ and $[m]$, respectively.
\end{lemma}
\begin{proof}
(i) The original inclusion \eqref{p} implies that $\exists s_j^\star \in (B_j\square D_j) (L_j u^\star - b_j)$ (which is equivalent to the second line of \eqref{pd}), such that  the first line of \eqref{pd} holds. In addition, it is also easy to recover \eqref{p} from \eqref{pd}.

(ii) To proceed with \eqref{pd}, based on a consensus form of ADMM \cite[P4-1]{admm_consensus} or  variable splitting suggested in \cite{salsa,csalsa}, the original variable $u^\star$ associated with $A_0$ is renamed as $u_0^\star$ and replicated into $(n-1)$ distinct surrogate variables $u_i^\star$ associated with $A_i$ for $i\in [n-1]$ (also known as {\it local agents} in signal processing communities), which yields the last line of \eqref{pds}. 

Then, by arbitrarily redistributing  the sums of $L_js_j^\star$ and $C_l u_0^\star$ over any $(A_i,u_i^\star)$, we rewrite the first line of \eqref{pd} as
\[
0  \in  \Big( \underbrace{  A_0 u_0^\star + \sum_{j\in \cJ_0}  L_j^* s_j^\star + \sum_{l\in \cP_0}  C_l u_0^\star - a}_{\owns -w_0^\star} \Big)
+ \sum_{i=1}^{n-1} \Big( \underbrace{ A_i u_i^\star + \sum_{j\in \cJ_i}  L_j^* s_j^\star + \sum_{l \in \cP_i}  C_l u_0^\star}_{\owns -w_i^\star}  \Big), 
\]
where the auxiliary variables $w_i^\star$ for $i=0,\ldots,n-1$ are introduced such that $\sum_{i=0}^{n-1} w_i^\star = 0$.  Finally, substituting $w_0^\star = - \sum_{i=1}^{n-1} w_i^\star $ into $(A_0,u_0^\star)$ yields \eqref{pds}.  
\end{proof}
\begin{remark}
Lemma \ref{l_pd}-(i) is essentially the same as the dual inclusion of \cite[Eq.(1.3)]{bot_jmiv_2014}, \cite[Problem 3.1]{arias_jota}. Our splitting strategy of Lemma \ref{l_pd}-(ii) relies on $B_j^{-1}$, which is fundamentally different from the way that \cite[Eq.(2)]{fjaa_pds} deals with the linear compositions. Our approach paves the way for developing algorithms in the spirit of Chambolle--Pock \cite[Algorithm 1]{cp_2011}.
\end{remark}

Based on \eqref{pds}, we propose the following prototype in an inclusion form, which will be shown to enjoy a very simple and clear graph-based interpretation. The involved notations, such as the multi-index $i_1,\ldots,i_q$ and the subset $\cI_{q,i_1,\ldots,i_{q-1}}$, will also be clarified therein. 

\be \label{proto}
\left\lfloor \begin{aligned}
& \hspace*{-.2cm} \textrm{(root)} \\
0 & \in   A_{0 } u_{ 0 }^{k+1}   +  \sum_{i_{1} =1}^{ |\cI_{1} | }
M_{i_1 }  (u_{0 }^{k+1}  - u_{i_1 }^{k}  +  w_{ i_1 }^k )  
+  \sum_{j \in\cJ_{0 } } L_{j}^* s_{j}^{k}  -a ; \\
0 & \in     B_{j_{0} }^{-1} s_{j_{0} } ^{k+1} + N_{j_0} (s_{j_{0} }^{k+1} - s_{j_{0} }^{k}) + D_{j_{0} }^{-1} s_{j_{0} }^{k}
- L_{j_{0} } u_0^{k+1} +b_{j_0},
\quad  \forall j_0 \in   \cJ_0 ; \\
& \hspace*{-.2cm} \textrm{(the general $q$-th level with
 $q\in [n-1]$)}  \\
0 & =  w_{i_1,\ldots,i_{q}}^{k+1} -w_{i_1,\ldots,i_{q} }^{k} + u_{i_1,\ldots,i_{q} }^{k}  - u_{i_1,\ldots,i_{q-1} }^{k+1}, \quad 
\forall i_q \in \big[\big| \cI_{q,i_1,\ldots, i_{q-1} } \big| \big] ; \\
0 &  \in  A_{i_1,\ldots,i_{q} } u_{ i_1,\ldots,i_{q} }^{k+1}   + 
M_{i_1,\ldots,i_{q} } (u_{i_1,\ldots,i_{q} }^{k+1}  - u_{ i_1,\ldots,i_{q-1} }^{k+1}  -   w_{ i_1,\ldots,i_{q} }^{k+1} )    \\
 & +  \sum_{i_{q+1} =1}^{ |\cI_{q+1,i_1,\ldots, i_q} | }
M_{i_1,\ldots,i_{q+1} }  (u_{i_1,\ldots,i_{q} }^{k+1}  - 
u_{i_1,\ldots,i_{q}, i_{q+1} }^{k}  +  w_{ i_1,\ldots,i_{q}, i_{q+1} }^k ) 
+  \sum_{l\in \cP_{i_1,\ldots,i_q}} C_l u_{i_1,\ldots,i_{q-1}}^{k+1}   \\ 
& +  \sum_{j \in\cJ_{i_1,\ldots, i_{q} } } L_{j}^* s_{j}^{k}   + 
\sum_{j\in \cJ_{c,i_1,\ldots,i_{q}}} L_j^* (s_{j}^{k+1} - s_j^k ) ,
\quad \forall i_q\in \big[\big| \cI_{q,i_1,\ldots, i_{q-1} } 
\big| \big] ; \\
0 & \in     B_{j_{q} }^{-1} s_{j_{q} } ^{k+1} + N_{j_q} (s_{j_{q} }^{k+1} - s_{j_{q} }^{k}) + D_{j_{q} }^{-1} s_{j_{q} }^{k}
- L_{j_{q} } u_{i_1,\ldots,i_{q} }^{k+1} +b_{j_q},
\quad  \forall j_q \in   \cJ_{i_1,\ldots,i_q} . 
\end{aligned} \right. 
\ee

Here, we make the following assumption on the {\it metrics} $M_{i_1,\ldots,i_{q} } $ and $ N_{j_q}$ in \eqref{proto}, which is assumed to hold throughout this paper, unless otherwise stated.
\begin{assumption}  \label{assume_MN}
{\rm (i)} $M_{i_1,\ldots,i_{q} } $ and $ N_{j_q}$ are linear, bounded, self-adjoint and uniformly positive definite;

{\rm (ii)} $M_{i_1,\ldots,i_{q} } $ and $ N_{j_q}$ have closed ranges. 
\end{assumption}
\begin{fact}
The prototype \eqref{proto} is well-posed under Assumption  \ref{assume_MN}.
\end{fact}
\begin{proof}
It is easy to see that $u_{ i_1,\ldots,i_{q} }^{k+1}  $ and $s_{j_{q} } ^{k+1}$ can be expressed as the warped/ metric resolvents  $J_{A_{ i_1,\ldots,i_{q} }}^{M_{ i_1,\ldots,i_{q} }+\sum_{i\in } M_{ i_1,\ldots,i_{q+1} }}$ and $J_{B_{j_q}^{-1}}^{N_{j_q}}$ evaluated at some point. Under Assumption \ref{assume_MN}, both resolvents \cite{plc_warped,fxue_jorc} are well-defined (i.e., single-valued and of full domain)  by Minty's theorem \cite[Theorem 21.1]{plc_book} or the extended \cite[Lemma 2.1]{fxue_jorc}.
\end{proof}

\subsection{Underlying tree-structured graph} \label{sec_tree}
There is a simple and flexible mechanism over a tree-structured graph
underlying the prototype \eqref{proto}. 

\subsubsection{A basic hierarchy of primal nodes} \label{sec_primal}
The equalities among the $n$ agents $u_i$ for $i=0,\ldots,n-1$ can be  expressed using at least exactly $(n-1)$ simple equalities, which is indicated by the  last line of \eqref{pds} as a special instance. Tree consisting of $n$ nodes and $(n-1)$ edges is an ideal graph structure for this minimal representation, where each agent $u_i$ defines a primal node, and each of $(n-1)$ simple equalities is expressed by an edge between two nodes, as shown in Fig. \ref{fig_whole}.
It is also suggested by \eqref{pds} that $(A_i,u_i)$ always appear as pairs. Unlike the existing works that usually represent the nodes by $A_i$ \cite{bredies_lift}, we denote them by $u_i$, since the edges are treated in our interpretation as the equalities and message passing  between the local agents $u_i$, rather than the operators $A_i$. 

The nodes are arranged  hierarchically to at most $(n-1)$ levels\footnote{The case of $(n-1)$ levels corresponds to a fully sequential graph, as shown in Fig. \ref{fig_decenter}.}, where $u_0\rightarrow u_{i_1} \rightarrow \cdots u_{i_1,\ldots,i_q}  \rightarrow \cdots $ is a typical lineage from the root $u_0$ to a leaf, where
 the $u_i$ for $i\in [n-1]$ are renamed by a path-based indexing principle as $u_{i_1,\ldots,i_{q}}$ in the $q$-th level, in order to explicitly record its lineage up to the root $u_0$. 
{Each node $u_i$ with $i\in [n-1]$ is uniquely identified by a multi-index $i_1,\ldots,i_q$ using one-based indexing (i.e., starting from 1) at each level. Specifically, $u_{i_1,\ldots,i_{q}}$ (in level-$q$)  refers to the $i_q$-th child of the node $u_{i_1,\ldots,i_{q-1}}$  (in level-$(q-1)$), recursively tracing back to the root $u_0$. For example, the node $u_{i_1,i_{2}}$ (in level-2) represents the $i_2$-th child of $u_{i_1}$ (in level-1), which itself is the $i_1$-th child of the root  $u_0$.
 
In the first level, the subset  $\cI_1$ collect the original indices in $[n-1]$ of all children $u_{i_1}$ of the root $u_0$, and $|\cI_1|$---the cardinality of $\cI_1$---denotes  the number of children of $u_0$. Thus, the level-1 contains the nodes $\{ u_{i_1} \}_{i_1=1}^{|\cI_1|}$.
For the level $q\ge 2$,  we use $\cI_{q,i_1,\ldots,i_{q-1}}$ to denote the original indices in $[n-1]$ of all the children of the node $u_{i_1,\ldots,i_{q-1}}$. The children of $u_{i_1,\ldots,i_{q-1}}$ are denoted as $u_{i_1,\ldots,i_{q-1},i_{q}}$ for $i_{q} =1,\ldots,   |\cI_{q,i_1,\ldots,i_{q-1}} | $. 
For example, any child of $u_{i_1,i_2}$ corresponds to $u_i$ for some $i\in \cI_{3,i_1,i_{2}} \subseteq [n-1]$ in the original indexing. Here, 3 indicates this node $u_i$ is in level-3. The number of children of  $u_{i_1,i_2}$ is the cardinality $|\cI_{3,i_1, i_{2}} |$.
In this way,  the correspondence between multi-indexing $i_1,\ldots,i_q$ and single index $i$ is given as $\big\{ u_{i_1,\ldots,i_q} \big\}_{i_q=1}^{|\cI_{q,i_1,\ldots,i_{q-1}}|}
= \big\{ u_i  \big\}_{i \in \cI_{q,i_1,\ldots,i_{q-1}} }$.  Refer to Remark \ref{rmk_notation} for further detailed explanations.
}

\begin{figure} [H]
\centering
\hspace*{-.1cm}
\scalebox{0.98} {
\begin{tikzpicture}[
    every node/.style = {
        draw,                   
        rounded corners=1mm,    
        minimum height=5mm,    
        inner xsep=2mm,         
        inner ysep=0.2mm,         
        font=\small             
    },
    level distance = 3cm,       
    sibling distance = 2cm      
]

\node (u0) at (0, 0) {$u_0$};

\node (u1) at (3, 0) {$u_{i_1}$}; 
\node (u2) at (6, 0)  {$u_{i_1,i_2}$}; 
\node[draw=none] at (7.5, 0)  {$\cdots$}; 
\node (u3) at (9, 0)  {$u_{i_1,\ldots,i_q}$}; 
\node[draw=none] at (11.5, 0)  {$\cdots$}; 
\draw (7.8,0) -- (u3);

\draw (u0) -- (u1)
node[draw=none, midway, above, sloped, text=gray] {$w_{i_1}$}   
node[draw=none, midway, below, sloped, text=gray] {$M_{i_1}$};

\draw (u1) -- (u2)
node[draw=none, midway, above, sloped, text=gray] {$w_{i_1,i_2}$}   
node[draw=none, midway, below, sloped, text=gray] {$M_{i_1,i_2}$};

\draw (u2) -- (7,0);
\draw (u3) -- (11.2,0)
node[draw=none, midway, above, sloped, text=gray] {$w_{i_1,\ldots,i_q}$}   
node[draw=none, midway, below, sloped, text=gray] {$M_{i_1,\ldots,i_q}$};

\node[draw=none, font=\small, text=gray] at (0, -.6) {$0$};
\node[draw=none, font=\small, text=gray] at (3, -.6) {$i_1\in\cI_1$};
\node[draw=none, font=\small, text=gray] at (6, -.6) {$i_2\in [|\cI_{2,i_1}|]$};
\node[draw=none, font=\small, text=gray] at (9, -.6) 
{$i_q\in [|\cI_{q,i_1,\ldots,i_{q-1}}|]$};

\node[draw=none, font=\small, text=gray] at (3, .6) {$\{u_{i}\}_{i\in \cI_1}$};
\node[draw=none, font=\small, text=gray] at (6, .6) {$\{u_{i}\}_{i\in \cI_{2,i_1}}$};
\node[draw=none, font=\small, text=gray] at (9, .6) {$\{u_{i}\}_{i\in \cI_{q,i_1,\ldots,i_{q-1}}}$};

\node[draw=none, font=\small, text=gray] at (0, -1) {root};
\node[draw=none, font=\small, text=gray] at (3, -1) {level-1};
\node[draw=none, font=\small, text=gray] at (6, -1) {level-2};
\node[draw=none, font=\small, text=gray] at (9, -1) {level-$q$};

\draw[dashed,thick] (-.7, -1.4)  -- (12, -1.4) ;
\node[dashed] (s1) at (0.5, -2.3) {$s_{j}$ for $j\in  \cJ_0$}; 
\node[dashed] (s2) at (3, -2.3) {$s_{j}$ for $j\in  \cJ_{i_1}$}; 
\node[dashed] (s3) at (6.1, -2.3) {$s_{j}$ for $j\in  \cJ_{i_1,i_2}$}; 
\node[dashed] (s4) at (9.6, -2.3) {$s_{j}$ for $j\in  \cJ_{i_1,\ldots,i_q}$}; 

\node[draw=none] at (-.1, -1.7)  {$L_j$}; 
\node[draw=none] at (3, -1.7)  {$L_j$}; 
\node[draw=none] at (6.1, -1.7)  {$L_j$}; 
\node[draw=none] at (9.5, -1.7)  {$L_j$}; 

\draw[dashed, bend right=30] (u0) to (s1);
\draw[dashed, bend right=30] (u1) to (s2);
\draw[dashed, bend left=30] (u2) to (s3);
\draw[dashed, bend left=30] (u3) to (s4);

\node[draw=none, font=\footnotesize, text=gray] at (.1, -3) {\bf correction};

\node[draw=none] (c1) at (2.3, -3.2) {$\displaystyle \sum_{j\in \cJ_{c,i_1}} L_j^* (s_{j}^{k+1} - s_j^k ) $};
\draw[dash dot, bend left=40] (u1) to (c1);

\node[draw=none] (c2) at (5.8, -3.2) {$\displaystyle \sum_{j\in \cJ_{c,i_1,i_2}} L_j^* (s_{j}^{k+1} - s_j^k ) $};
\draw[dash dot, bend right=40] (u2) to (c2);

\node[draw=none] (c3) at (10, -3.2) {$\displaystyle 
\sum_{j\in \cJ_{c,i_1,\ldots, i_{q}}} L_j^* (s_{j}^{k+1} - s_j^k ) $}; 
\draw[dash dot, bend right=30] (u3) to (c3);

\draw[dashed,thick] (-.7, 1)  -- (12, 1) ;
\node[draw=none, font=\footnotesize, text=gray] at (0.1, 1.6) {\bf cocoercive};

\node[draw=none] (p1) at (2.5, 1.6) {$C_l u_0^{k+1}$ for $l\in  \cP_{i_1}$}; 
\node[draw=none] (p2) at (5.6, 1.6) {$C_l u_{i_1}^{k+1}$ for $l\in  \cP_{i_1,i_2}$}; 
\node[draw=none] (p3) at (9.2, 1.6) {$C_l u_{i_1,\ldots,i_{q-1}}^{k+1}$ for $l\in  \cP_{i_1,\ldots,i_q}$}; 

\draw[dashed, bend right=30] (u1) to (p1);
\draw[dashed, bend right=30] (u2) to (p2);
\draw[dashed, bend right=30] (u3) to (p3);

\node[draw=none, font=\footnotesize, text=gray, align=left] at (.5, 0.7) {\bf primal  variables};
\node[draw=none, font=\footnotesize, text=gray] at (1.5, -1.8) {\bf dual variables};

\end{tikzpicture} } 
\vskip-.25cm
\caption{An ancestral line of a tree with dual assignments and  cocoercive loading. }
\label{fig_whole}
\vskip-.25cm
\end{figure}
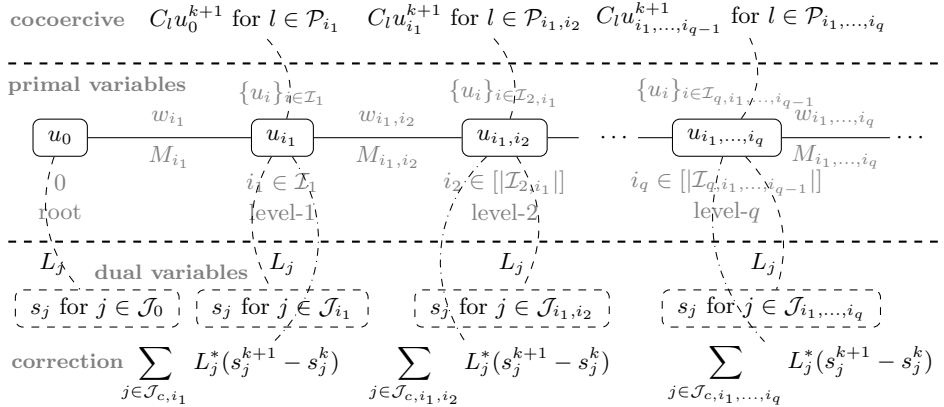

\subsubsection{Implication of the edges}
As mentioned above, the edge between $\allowbreak  u_{i_1,\ldots,i_{q-1}}$ and one of its children $u_{i_1,\ldots,i_q}$ represents the simple equality of $u_{i_1,\ldots,i_{q-1}} = u_{i_1,\ldots,i_{q}}$. This is much more general than, but equivalent to the last line of \eqref{pds}. For this parent-child pair, we introduce an auxiliary variable $w_{i_1,\ldots,i_{q}}$ and a weight  $M_{i_1,\ldots,i_{q}}$ between them, as shown in Fig. \ref{fig_whole}. Thus, any node $u_{i_1,\ldots,i_q}$ is linked to its parent (if any\footnote{In a tree graph, any node except for the root must have its own parent.}) $u_{i_1,\ldots,i_{q-1}}$ via $(w_{i_1,\ldots,i_q}, M_{i_1,\ldots,i_q})$ and to its children (if any\footnote{In a tree graph, any node except for the leaf must have its own children.}) $u_{i_1,\ldots,i_{q+1}}$ via $(w_{i_1,\ldots,i_{q+1}}, M_{i_1,\ldots,i_{q+1}})$. 
The pair $(w_{i_1,\ldots,i_q}, M_{i_1,\ldots,i_q})$  carries clear meaning in the $u_{i_1,\ldots,i_q}$-step of \eqref{proto}: $w_{i_1,\ldots,i_q}$ serves as a Lagrangian multiplier, while $M_{i_1,\ldots,i_q}$ acts as a weighting metric in the augmented Lagrangian technique. See Sect. \ref{sec_dynamical} for detailed explanations.

\subsubsection{Dual variables} 
In light of  the third line of \eqref{pds}, the dual element $(B_j^{-1}, D_j^{-1}, s_j)$ for $j\in [m]$ always appears together as a tuple.  After fixing the primal tree, we assign each node $u_{i_1,\ldots,i_q}$ with a number of $s_j$, such that these $s_j$ is updated synchronously with this primal node. Denote the index subset of the $s_j$ associated with  $u_{i_1,\ldots,i_q}$ by $\cJ_{i_1,\ldots,i_{q}} \subseteq [m]$. 
This assignment can be arbitrary, provided that each $s_j$ is assigned to some non-leaf primal node $u_i$ exactly once.
Again, the main skeleton of our graph consists of the primal variables only, with the dual part $s_j$ merely attached as an auxiliary component to the primal node.

\subsubsection{Cocoercive operators}
Each node $u_{i_1,\ldots,i_q}$  loads  a number of cocoercive operators $C_l$, which are evaluated at its parent  $u_{i_1,\ldots,i_{q-1}}$ due to causality.
This loading can also be arbitrary, provided each $C_l$ is loaded by some non-root primal node $u_{i_1,\ldots,i_{q}}$ exactly once (as the root $u_0$ lacks a parent to evaluate $C_l$).
Denote the index subset of the $C_l$ in $[p]$ associated with  $u_{i_1,\ldots,i_q}$ by $\cP_{i_1,\ldots,i_{q}}$. This is formally written for the node $u_{i_1,\ldots,i_q}$  (for $q\ge 1$) as  $\sum_{l\in \cP_{i_1,\ldots,i_q}} C_l u_{i_1,\ldots,i_{q-1}}^{k+1} $ in \eqref{proto}.

\subsubsection{Corrections} \label{sec_correction}
The assignments of those $s_j^k$ to a node $u_{i_1,\ldots,i_{q-1}}$ (i.e., $j\in \cJ_{i_1,\ldots,i_{q-1}}$) inevitably  cause the error $e_{i_1,\ldots,i_{q-1}}  = \sum_{j\in \cJ_{i_1,\ldots, i_{q-1}}} L_j^* (s_j^k -s_j^{k+1} )$ due to causality, then the correction taking the  form of  $\sum_{j\in \cJ_{i_1,\ldots, i_{q-1}}} L_j^* (s_j^{k+1} - s_j^k) (=- e_{i_1,\ldots,i_{q-1}})$ has to be performed  and shared by the updates of anyone of its children $u_{i_1,\ldots,i_q}$ for convergence issue.
Each child $u_{i_1,\ldots,i_{q}}$ is responsible for correcting a portion of the total error $e_{i_1,\ldots,i_{q-1}}$, precisely given as $\sum_{j\in \cJ_{c,i_1,\ldots,i_{q}}}  L_j^* (s_j^{k+1} - s_j^k)$, where the subscript $c$ denotes {\it correction}.  The distribution of the correction $-e_{i_1,\ldots,i_{q-1}}$ over the children $u_{i_1,\ldots,i_{q}}$ for $i_q=1,\ldots, |\cI_{q,i_1,\ldots,i_{q-1}}|$ can be arbitrary, as long as $\{ \cJ_{c,i_1,\ldots,i_q} \}_{i_q=1}^{|\cI_{q,i_1,\ldots,i_{q-1}}|}$ forms a partition of  $\cJ_{i_1,\ldots,i_{q-1}}$, i.e., 
\[
\sum_{i_q=1}^{ |\cI_{q,i_1,\ldots,i_{q-1}}| } \sum_{j\in \cJ_{c,i_1,\ldots,i_{q}}}  L_j^* (s_j^{k+1} - s_j^k)
=  \sum_{j\in \cJ_{i_1,\ldots,i_{q-1}}}  L_j^* (s_j^{k+1} - s_j^k)
=-e_{i_1,\ldots,i_{q-1}} . 
\]

\subsection{Dynamical mechanism} \label{sec_dynamical}
Before elaborating the dynamical mechanism behind the prototype \eqref{proto}, we make the following remark  concerning the abbreviated notations and formal consistency.
{\begin{remark} \label{rmk_notation}
{\rm (i)} In the remainder of this paper, we frequently use the shorthand $i_q$ to denote the full multi-index $i_1,\ldots,i_q$ for variables such as $A$, $u$, $w$ and $M$ (as well as the forthcoming variables $z$ and $C$) for brevity. Specifically, we write $u_{i_q}$ instead of $u_{i_1,\ldots,i_q}$, provided that omitting the leading indices $i_1,\ldots,i_{q-1}$ introduces no ambiguity, i.e., $u_{i_1,\ldots,i_q} := u_{i_q}$.
In other words, the notation $u_{i_q}$ must be associated with some (unique) $i_1\ldots,i_{q-1}$ in preceding levels, such that $u_{i_q} = u_{i_1,\ldots,i_{q-1},i_q}$, and $u_{i_1,\ldots,i_{q-1}} (:=u_{i_{q-1}})$ is the parent of  $u_{i_1,\ldots,i_{q}} (:=u_{i_{q}})$. For example, the shorthand notation $u_{i_3}$ implies that this node is in level-3, and there exist unique $i_1$ and $i_2$, such that $u_{i_3} = u_{i_1,i_{2},i_3}$. This $u_{i_3}$ is the $i_3$-th child of $u_{i_1,i_2}$. Again, $u_{i_1,i_2}$ can also be denoted as $u_{i_2}$ in short, which is the $i_2$-th child of $u_{i_1}$---the $i_1$-th child of the root $u_0$. In all, both notations of  $u_{i_1,\ldots,i_{q}}$ and $u_{i_{q}} $  carry exactly the same meaning and are used interchangeably depending on the context. 

{\rm (ii)} According to Sect. \ref{sec_tree}, we extend the definition settings to the boundary cases as follows, in order to ensure consistency with the internal node:
\begin{itemize}
\item {\rm Root:} When  $q=0$, we set $u_{i_1,\ldots,i_q} = u_0$,  $\cJ_{i_1,\ldots,i_q}=\cJ_0 $,
 $\cP_{i_1,\ldots,i_q}=\cP_0 =\varnothing$, $\cJ_{c,i_1,\ldots,i_q}= \varnothing$;
 
\item {\rm Leaf nodes:}  When  $u_{i_1,\ldots,i_q}$ is a leaf, $\cI_{q+1,i_1,\ldots,i_q} = \varnothing$, 
   $\cJ_{i_1,\ldots,i_q}= \varnothing$;

\item {\rm Children of root:}  When  $q=1$, we set
 $\cI_{q,i_1,\ldots,i_{q-1}} = \cI_1 $.
\end{itemize}

{\rm (iii)} The correspondence between single index $i \in \cI_{q,i_1,\ldots,i_{q-1}}$ and path-based multi-indexing $i_q \in [|\cI_{q,i_1,\ldots,i_{q-1}}|]$ has been defined in Sect. \ref{sec_primal}. In general, we adopt $i_q$ (which already implies the existence of unique $i_1,\ldots,i_{q-1}$ in preceding levels as mentioned in (i)) in the context of the parent-child relations within the tree structure, and use $i$ in the original indexing $[n-1]$ without the context of tree. We frequently switch between these two indexing systems where unambiguous.
\end{remark} }

For the general $q$-th level (with $q\ge 1$),  the update begins with computing the edge $w_{i_1,\ldots,i_{q}}^{k+1}$, which enforces the consensus constraint of both nodes it links:
\be \label{w}
0  = \underbrace{   w_{i_1,\ldots,i_{q}}^{k+1} -w_{i_1,\ldots,i_{q} }^{k}  }
_\textrm{self-update} + \underbrace{  u_{i_1,\ldots,i_{q} }^{k}  - u_{i_1,\ldots,i_{q-1} }^{k+1}  }_\textrm{consensus constraint}, \quad 
\forall i_q \in \big[\big| \cI_{q,i_1,\ldots, i_{q-1} } \big| \big],
\ee
where $ u_{i_1,\ldots,i_{q-1} }$ is the unique common parent of $ u_{i_1,\ldots,i_{q} }$ for $i_q=1,\ldots, |I_{q,i_1,\ldots,i_{q-1}}|$. 
This is quite a standard update of Lagrangian multiplier within the ADMM framework  \cite{salsa,csalsa}. 

Second,  the update of  $u_{i_q}$ is driven by consensus constraints with its parent and children. These equality constraints are associated
with Lagrangian multipliers  $w_{i_q}$ and  $w_{i_{q+1}}$,
 and weighting metrics $M_{i_q}$ and $M_{i_{q+1}}$,  respectively:
\be \label{u}
\begin{aligned}
0 & \in   A_{i_{q} } u_{ i_{q} }^{k+1}   + 
\underbrace{  M_{ i_{q} } \big(u_{ i_{q} }^{k+1}  - u_{  i_{q-1} }^{k+1}  -   w_{ i_{q} }^{k+1} \big) }
_\textrm{link to (unique) parent}   
+ \sum_{i_{q+1} =1}^{ |\cI_{q+1,i_1,\ldots, i_q} | }
\underbrace{  M_{ i_{q+1} }  \big( u_{ i_{q} }^{k+1}  - 
u_{ i_{q+1} }^{k}  +  w_{  i_{q+1} }^k \big) }
_\textrm{link to all children}  \\
& + \underbrace{  \sum_{l\in \cP_{i_1,\ldots,i_q}} C_l u_{ i_{q-1}}^{k+1} }
_\textrm{load of cocoercive} 
+ \underbrace{  \sum_{j \in\cJ_{i_1,\ldots, i_{q} } } L_{j}^* s_{j}^{k} }
_\textrm{load of dual }  + 
\underbrace{ \sum_{j\in \cJ_{c,i_1,\ldots,i_{q}}} L_j^* (s_{j}^{k+1} - s_j^k ) }_\textrm{correction for load by parent} ,
\quad \forall i_q\in \big[\big| \cI_{q,i_1,\ldots, i_{q-1} } \big| \big] .
\end{aligned}
\ee

This step is essentially a consensus-ADMM for the primal nodes (see the first line of \eqref{u}), incorporating the cocoercive by the form of forward-backward splitting \cite{plc} and dual elements following the style of Chambolle--Pock \cite{cp_2011}.
To see the consensus-ADMM spirit of  \eqref{w} and the first line of \eqref{u} more clearly,  let $A_{i_q} = \partial f_{i_q}$ for some proper, lower semi-continuous and convex function $f_{i_q}$. Then, the  first line of \eqref{u} is equivalent to:
\begin{align*}
u_{i_q}^{k+1} =& \arg\min_u f_{i_q} (u) - \big\langle u-u_{i_{q-1}}^{k+1} \big| M_{i_q} w_{i_q}^{k+1} \big\rangle
+\frac{1}{2} \big\| u-u_{i_{q-1}}^{k+1}  \big\|_{M_{i_q}}^2\\
+& \sum_{i_{q+1} =1}^{ |\cI_{q+1,i_1,\ldots, i_q} | }
\Big( \big\langle u-u_{i_{q+1}}^{k} \big| M_{i_{q+1}} w_{i_{q+1}}^{k} \big\rangle
+\frac{1}{2} \big\| u-u_{i_{q+1}}^{k}  \big\|_{M_{i_{q+1}}}^2
\Big).
\end{align*}
This is the well-known augmented Lagrangian formulation, stemming from the equality-constrained (i.e., consensus) minimization problem with respect to the current primal node $u_{i_q}$:
\[
\min_u f_{i_q} (u), \quad \textrm{s.t.\ } u= u_{i_{q-1}}\textrm{\ and\ } u= u_{i_{q+1}},\ \forall i_{q+1} \in [|\cI_{q+1,i_1,\ldots, i_q} | ].
\]
Here, $w_{i_q}$ acts as the Lagrangian multiplier associated with the constraint $u_{i_q} = u_{i_{q-1}}$, and is updated by quite a standard step \eqref{w}, the last term
$\frac{1}{2} \big\| u-u_{i_{q-1}}^{k+1}  \big\|_{M_{i_{q}}}^2$ is the augmentation equipped with the metric $M_{i_q}$. Both $w_{i_{q+1}}$ and $M_{i_{q+1}}$ play exactly the same roles for the constraint $u_{i_q} = u_{i_{q+1}}$.

Finally, the dual variable $s_{j_q}$ for $j_q\in \cJ_{i_1,\ldots,i_q} $ is updated in a self-driven manner:
\[
0  \in     B_{j_{q} }^{-1} s_{j_{q} } ^{k+1} +
\underbrace{  N_{j_q} (s_{j_{q} }^{k+1} - s_{j_{q} }^{k})}
_\textrm{self-evolution} + D_{j_{q} }^{-1} s_{j_{q} }^{k}
- \underbrace{  L_{j_{q} } u_{i_1,\ldots,i_{q} }^{k+1} }
_{\substack{\text{call-back for} \\ \text{previous load}}} 
+ b_{j_q},
\quad  \forall j_q \in   \cJ_{i_1,\ldots,i_q} . 
\]
In sharp contrast to \cite{fjaa_pds,tyc_pds}, our update of $s_{j_q}^k$ utilizes $B_{j_q}^{-1}$ following the style of Chambolle--Pock \cite[Algorithm 1]{cp_2011} and incorporates a forward step of $D_{j_q}^{-1}$ as in V\~{u} \cite[Theorem 3.1]{vu_2013}. This step is also a classic forward-backward splitting \cite{plc}, where the forward step using $D_{j_q}^{-1}$  is followed by a backward step of using resolvent of $B_{j_q}^{-1}$.

Summarizing all the above leads to the prototype \eqref{proto}, following the update order as
\[
\underbrace{  u_0^{k} \rightarrow \big( s_{j_0}^{k},}
_\textrm{root}  
\underbrace{ w_{i_1}^{k} \big) \rightarrow 
u_{i_1}^{k} 
\rightarrow   (s_{j_1}^k,  }_\textrm{level-1} 
 w_{i_2}^k) \rightarrow 
 \cdots  \rightarrow
 \big( s_{j_{q-1}}^{k}, 
 \underbrace{  w_{ i_{q}}^{k} \big)  \rightarrow
 u_{ i_q}^{k}  \rightarrow
 \big( s_{j_q}^{k},}_\textrm{level-$q$} 
 w_{ i_{q+1}}^{k} \big) 
\rightarrow  \cdots , 
\]
where $j_q\in \cJ_{i_1,\ldots,i_q}$. 
The protocol of constructing the tree  in Fig. \ref{fig_whole} is summarized below.

\noindent
\fbox{
    \parbox{0.95\linewidth}{
    {\bf Protocol}
\begin{itemize}
\item {\it Defining nodes:} The tuples of $(A_i, u_i)$ and $(B_j^{-1},D_j^{-1}, s_j)$ always appear together as the single units.  Any one of $\{u_i\}_{i=0}^{n-1}$ can be treated as the root $u_0$;
\item {\it Hierarchical structure:} The hierarchical arrangements of $u_i$ can be arbitrary, provided that each $u_i$ appears exactly once;
\item {\it Assigning dual to primal nodes:} Any non-leaf nodes $u_i$ can be associated with any finite number of dual variables $s_j$, provided that each $s_j$ is assigned exactly once;
\item {\it Loading cocoercive operators:} Any non-root nodes $u_i$ can  load  any finite number of cocoercive operators $C_l$, evaluated at the parent of $u_i$ (due to causality and parent-child link), provided that each $C_l$ is loaded exactly once;
\item {\it Corrections:} The error caused by assigning each $s_j$ (to some non-leaf node $u_i$) must be corrected exactly once, taking the form of $L_j^* (s_j^{k+1} -s_j^k)$ in the update of any child of this associated $u_i$.
\end{itemize}  } }

\vskip.2cm
The prototype \eqref{proto} is {\it level-synchronous}, where  $u_{i_q}$ within the same level-$q$ for $q=1,\ldots, |\cI_{q,i_1,\ldots,i_{q-1}}|$ are  updated synchronously, followed by the steps of those $s_j$ with $j\in \cJ_{i_1,\ldots,i_q}$ assigned to $u_{i_q}$. 
Moreover, according to \cite[Definition 2.3]{full} or \cite[Definition 2.6]{bredies_lift}, a tree structure essentially defines an {\it equivalent class} of frugal splitting algorithms, in a sense that they can be obtained with different metrics $M_i$ and $N_j$, or by swapping the order of $(A_i,u_i)$ and $(B_j^{-1}, D_j^{-1}, s_j)$. However, changing the tree structure, dual assignment or cocoercive loading gives fundamentally distinct schemes.

\subsection{Typical graph structures}
The underlying graph structure fundamentally shapes the distributed computing paradigm, often imposing a trade-off where centralization enables parallelism, whereas decentralization necessitates sequential processing to ensure consensus. 

\subsubsection{Centralized network}
Fig. \ref{fig_center} shows a centralized network realized by 
a 2-level (depth=1) tree, where all the $u_i$ except for the root $u_0$ are located in level-1 as leaves. This exactly represents the last line of \eqref{pds}: $u_0 = u_i$, $\forall i\in [n-1]$. Because the leaf nodes cannot be assigned to any dual variables, one has to assign all $s_j$ to the root $u_0$. The root $u_0$ acts as a central coordinator that computes the global sum across all nodes, broadcasting the aggregated result to the entire network. 
 
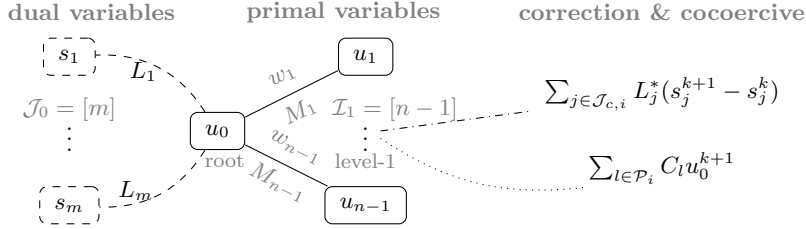
\begin{figure} [H]
\centering
\hspace*{-1cm}
\scalebox{0.98} {
\begin{tikzpicture}[
    every node/.style = {
        draw,                   
        rounded corners=1mm, 
        minimum height=5mm,  
        inner xsep=2mm,         
        inner ysep=0.2mm,          
        font=\small              
    },
    level distance = 3cm,    
    sibling distance = 2cm   
]

\node (u0) at (0, 0) {$u_0$};

\node (u1) at (2, 1) {$u_{1}$}; 
\node[draw=none] at (2, 0)  {$\vdots$}; 
\node (u2) at (2, -1)  {$u_{n-1}$}; 

\draw (u0) -- (u1)
node[draw=none, midway, above, sloped, text=gray] {$w_{1}$}   
node[draw=none, midway, below, sloped, text=gray] {$M_{1}$};

\draw (u0) -- (u2)
node[draw=none, midway, above, sloped, text=gray] {$w_{n-1}$}   
node[draw=none, midway, below, sloped, text=gray] {$M_{n-1}$};

\node[dashed] (s1) at (-2, 1) {$s_{1}$}; 
\node[draw=none] at (-2, 0)  {$\vdots$}; 
\node[dashed] (s2) at (-2, -1)  {$s_{m}$}; 
\node[draw=none] at (-1, .8)  {$L_1$}; 
\node[draw=none] at (-1.1, -.8)  {$L_m$}; 

\draw[dashed, bend right=30] (u0) to (s1);
\draw[dashed, bend left=30] (u0) to (s2);

\node[draw=none, font=\footnotesize, text=gray] at (0.1, -.4) {root};
\node[draw=none, font=\footnotesize, text=gray] at (2, -.4) {level-1};

\node[draw=none, font=\small, text=gray] at (2.4, .3) {$\cI_1 = [n-1] $};
\node[draw=none, font=\small, text=gray] at (-2, 0.3) {$\cJ_{0} = [m]$};

\node[draw=none, font=\small, text=gray] at (1.7, 1.6) {\bf primal variables};
\node[draw=none, font=\small, text=gray] at (-1.7, 1.6) {\bf dual variables};
\node[draw=none, font=\small, text=gray] at (6, 1.6) {\bf correction \& cocoercive};

\node[draw=none] (c1) at (6, .5) {$\sum_{j\in \cJ_{c,i}} L_j^*(s_j^{k+1} - s_j^{k})$}; 
\node[draw=none] (p1) at (6, -.5) {$\sum_{l\in \cP_{i}} C_l u_0^{k+1}  $}; 
\draw[dash dot] (2.2, 0) to (c1);
\draw[dotted, bend right=20] (2.2, -0.1) to (p1);

\end{tikzpicture} } 
\vskip-.25cm
\caption{A star-shaped tree for centralized computing. }
\label{fig_center}
\vskip-.25cm
\end{figure}

The corresponding centralized scheme is given as
{\small
\be \label{proto_center}
\left\lfloor \begin{aligned}
0 & \in  A_0 u_0^{k+1} + \sum_{i=1}^{n-1} M_i (u_0^{k+1}  - u_{i}^{k} 
+ w_{i}^k ) + \sum_{j =1}^m L_{j}^* s_{j}^{k}  - a ; \\
0 & =  w_{i}^{k+1} -w_{i}^{k} + u_{i}^{k}  - u_0^{k+1},
\quad \forall i \in [n-1] ; \\
0 & \in    B_{j}^{-1} s_{j} ^{k+1} + N_j (s_{j}^{k+1} - s_{j}^{k}) + D_{j}^{-1} s_{j}^{k}  - L_{j} u_0^{k+1} + b_j,
\quad \forall  j \in [m] ; \\
0 & \in   A_{i} u_{i}^{k+1}   + M_i (u_{i}^{k+1}  - u_0^{k+1}  -   w_{i}^{k+1} )  +  \sum_{ l \in \cP_{i} } C_l u_0^{k+1}  + \sum_{ j\in \cJ_{c,i} } L_{j}^* (s_{j}^{k+1} - s_{j}^{k}),
\quad \forall i\in [n-1] ,  \\
\end{aligned} \right. 
\ee}%
where $\{ \cJ_{c,i} \}_{i=1}^{n-1} $ and $\{ \cP_{i} \}_{i=1}^{n-1} $ form  partitions of $[m]$ and $[p]$, respectively. 

This  is generally considered undesirable due to default network settings, privacy concerns and cost issues \cite{fjaa_ring}.
 On the other hand, the fully parallel computing of $u_i$ for $i\in [n-1]$ is achieved at cost of the centralization.
In particular, if $m=0$ and $p=0$, our Fig. \ref{fig_center} reduces to the parallel-up/star graph \cite[Fig. 1-(c)]{fjaa_graph}, \cite[Table 1]{bredies_lift}. Note that our tree structure does not include the parallel-down graph \cite[Fig. 1-(d)]{fjaa_graph}, \cite[Table 1]{bredies_lift}, which is left to future work.

\subsubsection{Decentralized network}
This is realized by a tree with $n$-levels (depth = $(n-1)$), where each $u_i$ is located in the $i$-th level, as shown in Fig. \ref{fig_decenter}. This corresponds to a chain of equalities: $u_0=u_1$, $u_1=u_2$, ..., $u_{n-2} = u_{n-1}$ from the root $u_0$ to the leaf $u_{n-1}$. The dual variables $s_j$ can be assigned to any node $u_i$ except for the leaf $u_{n-1}$. The caused error must be corrected at the next node $u_{i+1}$---the only child of $u_i$.

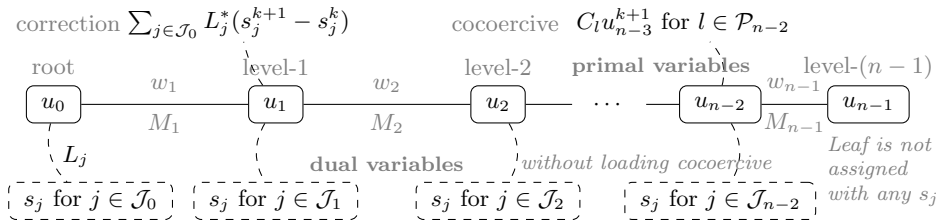
\begin{figure} [H]
\centering
\hspace*{-.1cm}
\scalebox{0.98} {
\begin{tikzpicture}[
    every node/.style = {
        draw,                  
        rounded corners=1mm, 
        minimum height=5mm, 
        inner xsep=2mm,         
        inner ysep=0.2mm,      
        font=\small             
    },
    level distance = 3cm,     
    sibling distance = 2cm  
]

\node (u0) at (0, 0) {$u_0$};

\node (u1) at (3, 0) {$u_{1}$}; 
\node (u2) at (6, 0)  {$u_2$}; 
\node[draw=none] at (7.5, 0)  {$\cdots$}; 
\node (u3) at (9, 0)  {$u_{n-2}$}; 
\node (u4) at (11, 0)  {$u_{n-1}$}; 

\draw (u0) -- (u1)
node[draw=none, midway, above, sloped, text=gray] {$w_{1}$}   
node[draw=none, midway, below, sloped, text=gray] {$M_{1}$};

\draw (u1) -- (u2)
node[draw=none, midway, above, sloped, text=gray] {$w_{2}$}   
node[draw=none, midway, below, sloped, text=gray] {$M_{2}$};

\draw (u2) -- (7,0);
\draw (8, 0) -- (u3);
\draw (u3) -- (u4)
node[draw=none, midway, above, sloped, text=gray] {$w_{n-1}$}   
node[draw=none, midway, below, sloped, text=gray] {$M_{n-1}$};

\node[draw=none, font=\small, text=gray] at (0, .5) {root};
\node[draw=none, font=\small, text=gray] at (3, .5) {level-1};
\node[draw=none, font=\small, text=gray] at (6, .5) {level-2};
\node[draw=none, font=\small, text=gray] at (11, .5) {level-$(n-1)$};

\node[dashed] (s1) at (0.5, -1.3) {$s_{j}$ for $j\in  \cJ_0$}; 
\node[dashed] (s2) at (3, -1.3) {$s_{j}$ for $j\in  \cJ_1$}; 
\node[dashed] (s3) at (6, -1.3) {$s_{j}$ for $j\in  \cJ_2$}; 
\node[dashed] (s4) at (9, -1.3) {$s_{j}$ for $j\in  \cJ_{n-2}$}; 

\node[draw=none] at (.3, -.7)  {$L_j$}; 

\draw[dashed, bend right=30] (u0) to (s1);
\draw[dashed, bend right=30] (u1) to (s2);
\draw[dashed, bend left=30] (u2) to (s3);
\draw[dashed, bend left=30] (u3) to (s4);

\node[draw=none] (p1) at (8.5, 1.1) {$C_l u_{n-3}^{k+1}$ for $l\in  \cP_{n-2}$}; 
\node[draw=none] (c1) at (2.5, 1.1) {$ \sum_{j \in  \cJ_{0} } L_j^* (s_j^{k+1} - s_j^{k} ) $}; 
\draw[dashed, bend right=30] (u3) to (p1);
\draw[dashed, bend left=10] (u1) to (c1);
\node[draw=none, font=\small, text=gray] at (6.1, 1.1) {cocoercive};
\node[draw=none, font=\small, text=gray] at (0.2, 1.1) {correction};

\node[draw=none, font=\footnotesize, text=gray, align=left] at (8.2, 0.5) {\bf primal  variables};
\node[draw=none, font=\footnotesize, text=gray] at (4.5, -.8) {\bf dual variables};
\node[draw=none, font=\footnotesize, text=gray] at (8, -.8) {\it without loading cocoercive};

\node[draw=none, font=\footnotesize, text=gray, align=left] at (11.2, -.9) {\it Leaf is not \\ \it assigned  \\ \it with any $s_j$ };
\end{tikzpicture} } 
\vskip-.25cm
\caption{A fully sequential tree for decentralized computing. }
\label{fig_decenter}
\vskip-.25cm
\end{figure}

The associated decentralized scheme is given as:
{\small
\be \label{proto_decenter}
\left\lfloor \begin{aligned}
0 & \in  A_0 u_0^{k+1} +  M_1 (u_0^{k+1}  - u_{1}^{k} 
+ w_{1}^k ) + \sum_{j \in \cJ_0}  L_{j}^* s_{j}^{k}  - a ;\\
0 & \in   A_{i} u_{i}^{k+1}   + M_i (u_{i}^{k+1}  - u_{i-1}^{k+1}  -   w_{i}^{k+1} ) + M_{i+1}  ( u_{i}^{k+1}  - u_{i+1}^{k}  +  w_{i+1}^{k} )   \\
& +   \sum_{l \in \cP_i}  C_{l} u_{i-1}^{k+1} 
+ \sum_{j \in \cJ_i}  L_{j}^* s_{j}^{k} 
  + \sum_{ j\in \cJ_{c,i}  } L_{j}^* (s_{j}^{k+1} - s_{j}^{k}),
\quad \forall i\in [n-2]  ; \\
0 & \in   A_{n-1} u_{n-1}^{k+1}   + M_{n-1} (u_{n-1}^{k+1}  - u_{n-2}^{k+1}  -   w_{n-1}^{k+1} ) 
 +   \sum_{l \in \cP_{n-1}}  C_{l} u_{n-2}^{k+1} 
  + \sum_{ j\in \cJ_{c,n-1}  } L_{j}^* (s_{j}^{k+1} - s_{j}^{k}); \\
0 & =  w_{i}^{k+1} -w_{i}^{k} + u_{i}^{k}  - u_{i-1}^{k+1}  ,
\quad \forall i\in [n-1]  ; \\
0 & \in  B_{j}^{-1} s_{j} ^{k+1} + N_j (s_{j}^{k+1} - s_{j}^{k}) 
+ D_{j}^{-1} s_{j}^{k}  - L_{j} u_{i}^{k+1} + b_j,
\quad \forall  j \in \cJ_i,  \quad \forall i\in \{0\} \mcup [n-2] ,  \\
\end{aligned} \right. 
\ee}%
where $\cJ_{c,i}=\cJ_{i-1}$ by the protocol mentioned above. 

In this message-passing mechanism, each agent is updated using only local information from its immediate parent and children.
However, the inherent sequential dependencies in the message flow necessitate a fully sequential computing.
In particular, if $m=0$ and $p=0$, our Fig. \ref{fig_decenter} reduces to the ordinary sequential graph \cite[Example 3.2]{fjaa_graph}, \cite[Table 1]{bredies_lift}.

\section{Fixed-point frugal splitting operator}
\subsection{Fixed point iterations}
The next proposition rewrites the $(u,w,s)$-scheme \eqref{proto} in the  recently popular fixed-point iterative style of \cite{fjaa_pds,campoy_product,yura_lift}, by a simple variable substitution. 
\begin{proposition}  \label{p_T}
Under Assumption \ref{assume_MN}, the prototype \eqref{proto} is equivalent to
\be \label{proto_T}
\begin{bmatrix}
z_{ i_q}^{k+1} \\ s_{j_q}^{k+1}  \end{bmatrix}
= \begin{bmatrix}
z_{ i_q}^{k} \\ s_{j_q}^{k}  \end{bmatrix}
+ \begin{bmatrix}
 u_{ i_{q}}^{k+1}  -  u_{ i_{q-1} }^{k+1}  \\ 
 N_{j_q}^{-1} \big(  v_{j_q}^{k+1} -  J_{N_{j_q} B_{j_q}}
 ( N_{j_q} s_{j_q}^k + v_{j_q}^{k+1} ) \big) 
  \end{bmatrix},\ 
\begin{array}{l}     \forall q\ge 1 \\ \forall j_q \in [m] 
\end{array}  
\ee
where  $v_{j_q}^{k+1} = L_{j_q} u_{  i_q}^{k+1} - D_{j_q}^{-1} s_{j_q}^{k} - b_{j_q}$ with $j_q\in \cJ_{i_1,\ldots,i_q}$, $u_{ i_q}^{k+1}$ for $q\ge 0$ are given by \eqref{uzs}.
Both schemes of \eqref{proto} and \eqref{proto_T} are linked via $z_{i_q}^k = u_{i_q}^k - w_{i_q}^k$.
\end{proposition}
\begin{proof}
Observe the $(u,w,s)$-scheme \eqref{proto}. Letting $z_{i_q}^k := u_{i_q}^k - w_{i_q}^k$ and  removing $w^k_{i_q}$, we obtain the equivalent $(u,z,s)$-scheme: 
\be \label{uzs}
\left\lfloor \begin{aligned}
& \hspace*{-.2cm} \textrm{(root)} \\
0 & \in   A_{0 } u_{ 0 }^{k+1}   +  \sum_{i_{1} =1}^{ |\cI_{1} | }
M_{i_1 }  \big( u_{0 }^{k+1}  - z_{i_1 }^{k}  \big)  
+  \sum_{j \in\cJ_{0 } } L_{j}^* s_{j}^{k}  -a ; \\
& \hspace*{-.2cm} \textrm{(the general $q$-th level with
 $q \ge 1$ and $\forall i_q \in \big[\big| \cI_{q,i_1,\ldots, i_{q-1} } \big| \big]$)}  \\
0 &  \in  A_{ i_{q} } u_{  i_{q} }^{k+1}   + 
M_{ i_{q} } \big( u_{ i_{q} }^{k+1} + 
z_{ i_{q} }^{k}  -2 u_{ i_{q-1} }^{k+1}   \big)     +  \sum_{i_{q+1} =1}^{ |\cI_{q+1,i_1,\ldots, i_q} | } M_{ i_{q+1} }  \big( u_{ i_{q} }^{k+1}  - 
z_{ i_{q+1} }^{k}  \big)  \\
& +  \sum_{l\in \cP_{i_1,\ldots,i_q}} C_l u_{ i_{q-1}}^{k+1}    
 +  \sum_{j \in\cJ_{i_1,\ldots, i_{q} } } L_{j}^* s_{j}^{k}   + 
\sum_{j\in \cJ_{c,i_1,\ldots,i_{q}}} L_j^* 
\big( s_{j}^{k+1} - s_j^k \big) ; \\
0 & =  z_{ i_{q}}^{k+1}  - 
z_{ i_{q}}^{k}  - u_{ i_{q}}^{k+1}   + u_{ i_{q-1} }^{k+1} . 
\end{aligned} \right. 
\ee

%

Here, the $s_j$-update is exactly the same as \eqref{proto}, which is omitted here to save page space. Then the proof is completed by rewriting the $s_j$-steps in \eqref{proto}  in terms of $J_{N_j B_j}$ via a metric extension of the well-known Moreau's resolvent identity \cite[Example 3.9-(i)]{plc_vu}. 
\end{proof}
\begin{remark}
We always use the $s_j$-update in the form of \eqref{proto} instead of \eqref{proto_T}. The latter is presented solely to maintain formal consistency with the fixed-point iteration of $z_{i_q}^k$.
\end{remark}

To align with prevailing algorithmic structures \cite{ryu_lift,yura_lift,fjaa_pds}, an additional relaxation step is often incorporated into \eqref{proto_T} as
\be \label{proto_T_relax}
\left\lfloor \begin{aligned}
\begin{bmatrix}
\tilde{z}_{ i_q}^{k} \\   \tilde{s}_{j_q}^{k}  \end{bmatrix}
& = \begin{bmatrix}
z_{ i_q}^{k} \\ s_{j_q}^{k}  \end{bmatrix}
+ \begin{bmatrix}
 u_{ i_{q}}^{k+1}  -  u_{ i_{q-1} }^{k+1}   \\ 
 N_{j_q}^{-1} \big( v_{j_q}^{k+1} -
 J_{N_{j_q} B_{j_q}}  ( N_{j_q} s_{j_q}^k + v_{j_q}^{k+1} )  \big)
  \end{bmatrix}, && \textrm{(prediction)} \\
\begin{bmatrix}
z_{ i_q}^{k+1} \\ s_{j_q}^{k+1}  \end{bmatrix}
& = \begin{bmatrix}
z_{ i_q}^{k} \\ s_{j_q}^{k}  \end{bmatrix}
+  \begin{bmatrix}
\theta_{ i_{q}} & 0   \\ 
0 & \zeta_{j_q}   \end{bmatrix} 
  \begin{bmatrix}
\tilde{z}_{ i_q}^{k} - {z}_{ i_q}^{k}    \\ 
 \tilde{s}_{j_q}^{k} - {s}_{j_q}^{k}   
  \end{bmatrix},  && \textrm{(relaxation)}\\
\end{aligned} \right. 
\ee
where $\theta_{i_1,\ldots,i_{q}}$ and $\zeta_{j_q} $ are the relaxation parameters. The updating order of \eqref{proto_T_relax} is:
\[
\underbrace{  u_0^{k} \rightarrow s_{j_0}^{k}}
_\textrm{root}  \rightarrow
\underbrace{ u_{i_1}^{k} \rightarrow 
\big(s_{j_1}^{k}, z_{i_1}^{k} \big) }_\textrm{level-1}  
\rightarrow \cdots  \rightarrow
\underbrace{ u_{i_1,\ldots,i_q}^{k} \rightarrow 
 \big( s_{j_{q}}^{k}, z_{i_1,\ldots,i_q}^k \big) }_\textrm{level-$q$}
 \rightarrow \cdots   .
\]

Before proceeding with the fixed-point iteration, we adopt the following notational convention: {\it
the subscript $i\in [n-1], j\in [m]$ of $(u_0^k, u_i^k, z_i^k, s_j^k)$ or $(u_0^k, u_i^k, w_i^k, s_j^k)$ is omitted for brevity. That is, we denote $(u_0^k, u_i^k, z_i^k, s_j^k)_{i\in [n-1], j\in [m]} := (u_0^k, u_i^k, z_i^k, s_j^k)$ and $(u_0^k, u_i^k, w_i^k, s_j^k)_{i\in [n-1], j\in [m]} := (u_0^k, u_i^k, w_i^k, s_j^k)$.}

The relaxed scheme \eqref{proto_T_relax} gives a representation of the {\it apparent} fixed-point operator $\cT_{uzs, \theta,\zeta}$ parametrized by $(A_{i_q}, B_{j_q}^{-1}, D_{j_q}^{-1}, C_l, L_{j_q}, M_{i_q}, N_{j_q}, \theta_{i_q}, \zeta_{j_q})$,  defined on $\cH\times \cH^{n-1} \times \cH^{n-1} \times \prod_{j=1}^m \cK_j$. Under the aforementioned convention, \eqref{proto_T_relax} can be expressed as  $(u_0^{k+1}, u_i^{k+1}, z_i^{k+1}, s_j^{k+1}) = \cT_{uzs,\theta,\zeta} (u_0^{k}, u_i^{k}, z_i^{k}, s_j^{k}) $,  where $\cT_{uzs,\theta,\zeta} = \cI - \begin{bmatrix}
1 & 0&0&0 \\   0& 1&0&0 \\
0&0 & \theta_i & 0 \\ 
0&0 & 0 & \zeta_j \end{bmatrix} (\cI -\cT_{uzs})$,
 $\cT_{uzs}$ denotes the first prediction step of \eqref{proto_T_relax}:  
$(u_0^{k+1}, u_i^{k+1}, \allowbreak \ztilde_i^{k}, \stilde_j^{k}) = \cT_{uzs} (u_0^{k}, u_i^{k}, z_i^{k}, s_j^{k}) $. 

Observe that all the primal variables $u_i^{k+1}$ depend entirely on $z_i^k$ and $s_j^k$. As mere intermediate variables, they do not actually participate in the {\it reduced} fixed-point iteration:  $(z_i^{k+1}, s_j^{k+1}) = \cT_{zs,\theta,\zeta} (z_i^{k}, s_j^{k})$. Here the fixed-point operator $\cT_{zs}$ is  defined on a reduced dimensional space $\cH^{n-1} \times \prod_{j=1}^m \cK_j $, which is given as $\cT_{zs,\theta,\zeta} = \cI - \begin{bmatrix}
\theta_i & 0 \\ 0 & \zeta_j \end{bmatrix} (\cI -\cT_{zs})$, where $ \cT_{zs}$ denotes the prediction step: 
$(\ztilde_i^{k}, \stilde_j^{k}) = \cT_{zs} (z_i^{k}, s_j^{k})$. 

The following fact can be readily verified.
\begin{fact} \label{f_T}
$ \Fix \cT_{uzs,\theta,\zeta} = \Fix \cT_{uzs}$ and  $ \Fix \cT_{zs,\theta,\zeta} = \Fix \cT_{zs}$, $\forall \theta_{i_q}, \zeta_{j_q} \ne 0$.
\end{fact}

\subsection{Fixed point encoding}
Proposition \ref{p_encoding} reveals the relationship between the fixed points of $\cT_{zs,\theta,\zeta}$ \eqref{proto_T_relax} and solutions to the problem \eqref{p}.
\begin{proposition} \label{p_encoding}
Given the fixed point operator $\cT_{zs,\theta,\zeta}$ (with $\theta_{i_q}, \zeta_{j_q} \ne 0$) defined by \eqref{proto_T_relax} , the following assertions hold:

{\rm (i)} If $(z_i^\star, s_j^\star) \in \Fix \cT_{zs,\theta,\zeta} $, then  $(u_0^\star,u_i^\star,z_i^\star, s_j^\star) \in \Fix \cT_{uzs,\theta,\zeta}$, where $u_0^\star = u_i^\star =  J_{A_0}^{\sum_{i_1=1}^{ |\cI_1|} M_{i_1}} 
\big( \sum_{i_1=1}^{ |\cI_1|} M_{i_1} z_{i_1}^\star -
 \sum_{j \in\cJ_{0 } } L_{j}^* s_{j}^\star  + a  \big)$, $\forall i\in [n-1]$, and $u_0^\star$ solves \eqref{p}.

{\rm (ii)} If $u^\star$ is a solution to \eqref{p}, then 
  $\exists (s_j^\star)_{j\in [m]} \in \prod_{j=1}^m \cK_j $,  
 $\exists (w_i^\star)_{i \in [ n-1] } \in \cH^{n-1}$, such that
 $(u_0^\star, u_i^\star, u_0^\star - w_i^\star, s_j^\star) \in \Fix \cT_{uzs,\theta,\zeta}$, where $u_0^\star = u_i^\star =u^\star$, and   $(u^\star - w_i^\star, s_j^\star) \in \Fix \cT_{zs,\theta,\zeta}$. 
\end{proposition}
\begin{proof}
(i) We develop $(z_i^\star, s_j^\star) \in \Fix \cT_{zs,\theta,\zeta} \Longleftrightarrow (z_i^\star, s_j^\star) \in \Fix \cT_{zs}$ by Fact \ref{f_T} $\Longleftrightarrow (z_i^\star, s_j^\star)$ is the fixed point of \eqref{uzs} by Proposition \ref{p_T}. 
The last line of \eqref{uzs} implies $u_0^\star = u_i^\star$, $\forall i \in [n-1]$, which is $u_0^\star =  J_{A_0}^{\sum_{i_1=1}^{ |\cI_1|} M_{i_1}} 
\big( \sum_{i_1=1}^{ |\cI_1|} M_{i_1} z_{i_1}^\star -
 \sum_{j \in\cJ_{0 } } L_{j}^* s_{j}^\star  + a  \big)$ by the first line of  \eqref{uzs}.  Then, the apparent fixed-point operator $\cT_{uzs}$ and $\cT_{uzs,\theta,\zeta}$ is recovered: 
  $(u_0^\star, u_i^\star, z_i^\star, s_j^\star) \in \Fix \cT_{uzs,\theta,\zeta}$, where $u_0^\star = u_i^\star =u^\star$.
 
Furthermore, by equivalence between \eqref{proto} and \eqref{uzs} via the link of $z_{i_q}^k = u_{i_q}^k - w_{i_q}^k$ (cf. Proposition \ref{p_T}), we conclude that  $(u_0^\star, u_{i_q}^\star, 
u_{i_q}^\star - z_{i_q}^\star, s_j^\star)_{q\ge 1, j\in [m]}$ (with $u_0^\star =u_{i_q}^\star$, $\forall q\ge 1$) is a fixed point of \eqref{proto}, i.e., 
{\small
\be \label{www}
\left\lfloor \begin{aligned}
0 & \in   A_{0 } u_{ 0 }^\star  +  \sum_{i_{1} =1}^{ |\cI_{1} | }
M_{i_1 }   w_{ i_1 }^\star   
+  \sum_{j \in\cJ_{0 } } L_{j}^* s_{j}^\star  -a ; \\
0 &  \in  A_{ i_{q} } u_{ 0}^\star  - 
M_{ i_{q} } w_{  i_{q} }^\star +  \sum_{i_{q+1} =1}^{ |\cI_{q+1,i_1,\ldots, i_q} | } M_{ i_{q+1} }  w_{  i_{q+1} }^\star 
+  \sum_{l\in \cP_{i_1,\ldots,i_q}} C_l u_{0}^\star 
 +  \sum_{j \in\cJ_{i_1,\ldots, i_{q} } } L_{j}^* s_{j}^\star ,\ 
 \forall q \ge 1 ; \\
0 & \in   \big( B_{j }^{-1} + D_{j }^{-1} \big) s_{j } ^\star 
- L_{j } u_0^\star +b_{j},\quad  \forall j \in [m]. \\
\end{aligned} \right. 
\ee}%
Substituting $s_j^\star \in (B_j\square D_j) (L_{j } u_0^\star -b_{j})$ into and summing up all the $A_i$-inclusions yields that $u_0^\star$ solves \eqref{p}, as all the terms containing  $w_{i_q}^\star$ cancel out.

(ii) If $u^\star$ is a solution to \eqref{p}, then by  Lemma \ref{l_pd}-(ii),  $\exists (u_i^\star)_{i=0,\ldots,n-1} \in \cH^n$,
 $\exists (y_i^\star)_{i \in [n-1] } \in \cH^{n-1}$,
   $\exists (s_j^\star)_{j\in [m]} \in \prod_{j=1}^m \cK_j $,  
such that      $u_0^\star = u_i^\star (:=u^\star)$ ($\forall i\in [n-1]$), 
$ 0  \in A_i u_0^\star + y_i^\star + \sum_{l \in \cP_i}  C_l u_0^\star 
+ \sum_{j \in \cJ_i}  L_j^* s_j^\star$ , $\forall i\in [n-1]$.  
Comparing this inclusion with the middle line of \eqref{www}, we obtain the following  linear system of equations of order $(n-1)$:
$ 0  =  y_i^\star +  M_{ i_{q} } w_{  i_{q} }^\star - 
 \sum\nolimits_{i_{q+1} =1}^{ |\cI_{q+1,i_1,\ldots, i_q} | } M_{ i_{q+1} }  w_{  i_{q+1} }^\star$,  
$ \forall i\in \cI_{q,i_1,\ldots,i_{q-1}}$ and $\forall q\ge 1 $.
Solving the above linear system by induction from the leaves to the root, we observe that $w_i^\star$ is uniquely determined by $ (y_i^\star)_{i \in [n-1] }$, since each $M_{i_q}$ is invertible under Assumption \ref{assume_MN}. Since \eqref{www} is the fixed point version of \eqref{proto},   $(u_0^\star, u_{i_q}^\star, w_{i_q}^\star, s_j^\star)_{q\ge 1, j\in [m]}$ is also a fixed point of the prototype \eqref{proto}. Finally, the desired result is obtained by Proposition \ref{p_T}.
\end{proof}

From Proposition \ref{p_encoding} follows Corollary \ref{c_existence}, which  plays a pivotal role in the subsequent convergence analysis, as demonstrated in Proposition \ref{p_ine} and Theorem \ref{t_zs}. 
\begin{corollary} \label{c_existence}
$\Fix \cT_{uzs,\theta,\zeta}\ne\varnothing$ and
$\Fix \cT_{zs,\theta,\zeta}\ne\varnothing$.
\end{corollary}
\begin{proof}
It has been assumed throughout this work that the solution set of \eqref{p} is non-empty. This is equivalent to the existence of the fixed points of $\Fix \cT_{uzs,\theta,\zeta}$ and $\Fix \cT_{zs,\theta,\zeta}$ by Proposition \ref{p_encoding}.
\end{proof}

Finally, we conclude
\begin{theorem} \label{t_encoding}
Given the problem \eqref{p}, the pair $(\cT_{zs,\theta,\zeta}, \cS)$

{\rm (i)} forms a fixed point encoding, where the fixed point operator $\cT_{zs,\theta,\zeta}$ is defined as  \eqref{proto_T_relax}, the solution mapping is $\cS = 
J_{A_0}^{\sum_{i_1=1}^{ |\cI_1|} M_{i_1}} 
\big( \sum_{i_1=1}^{ |\cI_1|} M_{i_1} z_{i_1}^\star -
 \sum_{j \in\cJ_{0 } } L_{j}^* s_{j}^\star  + a  \big)$; 
 
{\rm (ii)} is a frugal primal-dual parametrized resolvent splitting with $(n-1,m)$-fold minimal lifting.
\end{theorem}
\begin{proof}
(i) Clear from Proposition \ref{p_encoding} and by definition of fixed point encoding \cite{ryu_lift}, \cite[Definition 8]{fjaa_pds}, \cite[Definition 2.3]{yura_lift}. 

(ii) We say \eqref{uzs} and \eqref{proto_T_relax} are frugal by \cite[Definition 11]{fjaa_pds}, in a sense that the evaluation of parametrized resolvent of $B_j^{-1}$ is essentially equivalent to that of $B_j$ due to Moreau's decomposition identity \cite[Example 3.9-(i)]{plc_vu}. 

Regarding the lifting number, the active variables of \eqref{proto_T_relax} are clearly $(z_i,s_j)$, implying that the minimal number of splitting method for solving \eqref{p} is at most $(n-1,m)$-fold \cite[Definitions 9 and 10]{fjaa_pds}. On the other hand, within the class of frugal parametrized resolvent splitting considered in \cite{fjaa_pds}, any splitting capable of solving all intances of the class \eqref{p} requires at least $(n-1,m)$-fold lifting.  Indeed, this class contains the subclass $p=0$ and $D_j=\cN_{\{0\}}$, for which the lower bound of \cite[Theorem 5]{fjaa_pds} applies. 
Consequently, for all instances of the general problem \eqref{p}, the minimal lifting number should be bounded below by the $(n-1,m)$-fold. Then the result follows. 
Here, we emphasize that the minimal lifting number always holds for any tree structure with arbitrary dual assignment and cocoercive loading, provided that the graph design adheres to the protocol mentioned above.
\end{proof}

\section{Convergence analysis} \label{sec_con}
Our analysis heavily relies on the reformulation of \eqref{proto_T_relax} in a product Hilbert space, stated in Lemma \ref{l_ppm}. Theorem \ref{t_zs} shows the weak convergence of $\big\{ (z_i^k,s_j^k) \big\}_{k\in\N}$ based on the somewhat averagedness of $\cT_{zs,\theta,\zeta}$ (cf. Proposition \ref{p_ine}) and  demiclosedness of $\cI - \cT_{zs,\theta,\zeta}$ (cf. Corollary \ref{c_demi}). Proposition \ref{p_rate} further shows the $o(1/k)$-rate of asymptotic regularity. Note that many detailed long proofs are postponed to the Appendices.

\subsection{Basic ingredients}
This part presents several basic results of the relaxed scheme \eqref{proto_T_relax}, that will be useful for convergence analysis. 

First, several facts of cocoercive operator are necessary. 
\begin{fact} \label{f_coco}
Let  $C_l: \cH\mapsto \cH$ for $l=1,\ldots,p$ (resp. $C: \cH\mapsto \cH$) be $\beta_l$ (resp. $\beta$)-cocoercive. Then the following assertions hold:

{\rm (i)} $\sum_{l=1}^p C_l$ is $\frac{1}{\sum_{l=1}^p \beta_l^{-1}}$-cocoercive;

{\rm (ii)} $\langle Ca-Cb | c-d \rangle \ge  - \frac{1}{4\beta} 
\|   (c- a) - (d -b)  \|^2$, $\forall a,b,c,d\in \cH$.
\end{fact}
\begin{proof}
(i) Taking $L=I$ in \cite[Proposition 4.12]{plc_book} completes the proof.

(ii) We derive $\langle Ca-Cb | c-d \rangle   = \langle Ca-Cb | a-b  \rangle + \langle Ca-Cb |  c- a -d +b \rangle
\ge \beta \| Ca-Cb \|^2 - \beta \| Ca-Cb \|^2 - \frac{1}{4\beta}   \|   (c- a) - (d -b)  \|^2  =  - \frac{1}{4\beta} 
\|   (c- a) - (d -b)  \|^2$.
\end{proof}

Lemma \ref{l_ppm} reformulates the prediction step of \eqref{proto_T_relax} in a product Hilbert space, which facilitates the subsequent convergence analysis. We adopt the following notational convention.
\begin{remark} \label{rmk_notation2}
Let $(u_{i_q}, z_{i_q}, s_{j})$ with specific index $i_q$ and the corresponding $j\in \cJ_{i_1,\ldots,i_q}$ be a segment extracted from the full vector $(u_0, u_{i_q}, z_{i_q}, s_{j})_{i\in[n-1], j\in [m]}$.  Applying a block operator $\cA$ to this segment, we denote the elements of $\cA$ as:
{\small \[
\begin{bmatrix}
\cA_{u_{i_{q-1}} \rightarrow u_{i_{q-1}} } & 
\cA_{u_{i_{q-1}} \rightarrow u_{i_{q}} } & 
* & * & * & *  \\
\cA_{u_{i_{q}} \rightarrow u_{i_{q-1}} } & 
\cA_{u_{i_{q}} \rightarrow u_{i_{q}} } & 
\cA_{u_{i_{q}} \rightarrow u_{i_{q+1}} } & 
\cA_{u_{i_{q}} \rightarrow z_{i_{q}} } & 
\cA_{u_{i_{q}} \rightarrow z_{i_{q+1}} } & 
\cA_{u_{i_{q}} \rightarrow s_{j} }   \\
* & *  & \cA_{u_{i_{q+1}} \rightarrow u_{i_{q+1}} } & 
* & \cA_{u_{i_{q+1}} \rightarrow z_{i_{q+1}} } & *  \\
* & \cA_{z_{i_{q}} \rightarrow u_{i_{q}} } & *  & 
\cA_{z_{i_{q}} \rightarrow z_{i_{q}} } &  * & *  \\
* & \cA_{z_{i_{q+1}} \rightarrow u_{i_{q}} } 
& \cA_{z_{i_{q+1}} \rightarrow u_{i_{q+1}} }   & 
  * & \cA_{z_{i_{q+1}} \rightarrow z_{i_{q+1}} } & *  \\
* &   \cA_{s_j \rightarrow u_{i_{q}} } & * & * & *  
&  \cA_{s_j \rightarrow s_{j} } 
   \end{bmatrix}
 \begin{bmatrix}
u_{i_{q-1}} \\ u_{i_{q}} \\ u_{i_{q+1}} \\ z_{i_{q}} \\
z_{i_{q+1}} \\ s_j  \end{bmatrix}.
\]}%
Here the subscripts on the left and right of the arrow $\rightarrow$ denote the row and column indices, respectively. For example, the elements in the second row of $\cA$ are $\cA_{u_{i_q}\rightarrow ?}$, while the fourth column takes the form $\cA_{? }\rightarrow  z_{i_q}$. The symbol * in the above matrix serves as a placeholder for omitted entries. 
\end{remark} 

\begin{lemma} \label{l_ppm}
The prediction step of \eqref{proto_T_relax} is recast as $0  \in \cA \tilde{x}^{k}  + \cQ (\tilde{x}^{k} - x^k) + r^k$, defined in a product Hilbert space $\cH\times  \cH^{n-1} \times \cH^{n-1} \times \prod_{j=1}^m \cK_j$, where 
\begin{itemize}
\item $x = (u_0,u_i,z_i,s_j) \in 
\cH\times  \cH^{n-1} \times \cH^{n-1} \times \prod_{j=1}^m \cK_j$.

\item  The block operator $\cA$ is defined as follows (cf. Remarks \ref{rmk_notation} and \ref{rmk_notation2}):
$\cA_{u_{0}\rightarrow u_{0}}: \cH\mapsto 2^\cH: u_0\mapsto  A_{0} u_0 -a$, 
$\cA_{u_{i_q}\rightarrow u_{i_q}} = A_{i_q}$, $\forall q\ge 1$;
$\cA_{u_{i_{q}}\rightarrow u_{i_{q+1}}} = M_{i_{q+1}}$, 
$\cA_{u_{i_{q+1}}\rightarrow u_{i_{q}}} = -M_{i_{q+1}}$, $\forall q\ge 0$,  $\forall i_{q+1} \in \cI_{q+1,i_1,\ldots,i_{q}}$;
$\cA_{u_{i_{q}}\rightarrow z_{i_{q}}} = M_{i_{q}}$, 
$\cA_{z_{i_{q}}\rightarrow u_{i_{q}}} = - M_{i_{q}}$,  $\forall q\ge 1$; 
$\cA_{u_{i_{q}}\rightarrow z_{i_{q+1}}} = - M_{i_{q+1}}$, 
$\cA_{z_{i_{q+1}} \rightarrow u_{i_{q}} } =  M_{i_{q+1}}$,  $\forall q\ge 0$; 
$\cA_{u_{i_{q}}\rightarrow s_{j} } = L_{j}^*$, 
$\cA_{ s_{j}\rightarrow u_{i_{q}}} = -L_{j}$,   $\forall q\ge 0$,  $\forall j\in \cJ_{i_1,\ldots,i_q}$;
$\cA_{s_{j}\rightarrow s_{j} } = B_{j}^{-1}$, $\forall j\in [m]$.
All other unspecified entries are zero.

\item The block preconditioner $\cQ$ is given as: $\cQ_{z_{i_{q}}\rightarrow z_{i_q}} = M_{i_q}$,  $\forall q\ge 1$;  $\cQ_{s_{j}\rightarrow s_{j} } = N_{j}$, $\forall j\in [m]$. All of other entries unspecified above are zeros. 
\item The residual term $r^{k}$ is defined as $r^{k} = \big( r_{u_0}^{k}, r_{u_{i}}^{k}, r_{z_{i}}^{k}, r_{s_j}^{k} \big)_{i\in [n-1],   j\in [m]}$, where $ r_{z_{i}}^{k} = 0$, $\forall i\in [n-1]$; $r_{s_j}^{k} = D_{j}^{-1} s_{j}^{k} + b_{j}$, $\forall j\in [m]$; and $ r_{u_{i}}^{k} $ is rearranged by the tree nodes $r_{u_{i_q}}^k$, which is given as $\forall i \in \cI_{q, i_1,\ldots,i_{q-1}}$ and  $\forall q\ge 0$ (following Remarks \ref{rmk_notation} and \ref{rmk_notation2}):
\[
r_{u_{i_q}}^{k} = 
\sum_{j\in \cJ_{i_1,\ldots,i_{q}}}  L_{j}^* ( s_{j}^{k} -  
\stilde_{j}^{k} )
 + \sum_{j\in \cJ_{c,i_1,\ldots,i_{q}}} 
 L_{j}^* ( \stilde_{j}^{k} -  s_{j}^{k} )     
 + \sum_{l\in \cP_{i_1,\ldots,i_{q}}} C_l  u_{ i_{q-1}}^{k+1} . 
\]
\end{itemize}
\end{lemma} 
\begin{proof}
Considering the prediction step of \eqref{proto_T_relax},  the $u_i$-update, by variable substitution, becomes
\[
\begin{aligned}
0 &  \in  A_{ i_{q} } u_{  i_{q} }^{k+1}   + 
M_{ i_{q} } (\tilde{ z}_{ i_{q} }^{k} -  u_{  i_{q-1} }^{k+1}   )   
 +  \sum_{i_{q+1} =1}^{ |\cI_{q+1,i_1,\ldots, i_q} | }
M_{ i_{q+1} }  (u_{ i_{q+1} }^{k+1}  - \tilde{  z}_{ i_{q+1} }^{k}  ) 
+  \sum_{l\in \cP_{i_1,\ldots,i_q}} C_l u_{ i_{q-1}}^{k+1}   \\ 
& +  \sum_{j \in\cJ_{i_1,\ldots, i_{q} } } L_{j}^* s_{j}^{k}   + 
\sum_{j\in \cJ_{c,i_1,\ldots,i_{q}}} L_j^* (
\tilde{s}_{j}^{k} - s_j^k ) ,
\quad \forall i_q\in \big[\big| \cI_{q,i_1,\ldots, i_{q-1} } 
\big| \big].  \\
\end{aligned} 
\]
Then, it is easy to reexpress in matrix language according to $x = (u_0,u_i,z_i,s_j)$ as Lemma \ref{l_ppm}. Notice that the constant vector $a$ has been absorbed in the first diagonal element $\cA_{u_0\mapsto u_0}$ associated with the root $u_0$. 
\end{proof}

\begin{corollary} \label{c_AQ}
Given $\cA$ and $\cQ$ defined in Lemma \ref{l_ppm}, then the following assertions hold:

{\rm (i)} $\cA$ is maximally monotone; 

{\rm (ii)} $\cQ$ is diagonal (and thus self-adjoint), degenerate (positive semi-definite), and has closed range. 
\end{corollary}
\begin{proof}
(i) The block operator $\cA$ exhibits a typical (diagonal) {\it monotone} + (off-diagonal) {\it skew} structure \cite{arias_2011,fxue_optl}, and thus is maximally monotone by \cite[Lemma 2.6]{fxue_fail}.

(ii) All off-diagonal elements of $\cQ$ are zero and the diagonal element $\cQ_{u_{i_q} \rightarrow u_{i_q}} = 0$, $\forall q\ge 0$. Furthermore, $M_{i_q}$ and $N_j$ fulfill Assumption \ref{assume_MN}. Thus (ii) follows.
\end{proof}

To associate with the residual term $r^k$ given in Lemma \ref{l_ppm}, we define the residual mapping as $\cR: \cH^{2n-1} \times \prod_{j=1}^m \cK_j \mapsto   \cH^{2n-1} \times \prod_{j=1}^m \cK_j: x\mapsto r = \cR x$, and in particular, $r^k = \cR x^k$. Unlike $\cA$ and $\cQ$ explicitly expressed in Lemma \ref{l_ppm}, this $\cR$ in our reformulation is highly implicit due to the intricate relationship between $\xtilde^k$ and $x^k$. Nevertheless, the prediction step of \eqref{proto_T_relax} can be rewritten as $\xtilde = \cT_{uzs} x := (\cA +\cQ)^{-1} (\cQ - \cR)  x $. For any  $x_1,x_2 \in \cH^{2n-1} \times \prod_{j=1}^m \cK_j $, define 
\be \label{delta}
\Delta :=\Delta(x_1,x_2) = \big\langle \cR x_1 - \cR x_2 \big| \cT_{uzs} x_1 - \cT_{uzs} x_2 \big\rangle = 
\big\langle \cR x_1 - \cR x_2 \big| \xtilde_1 -  \xtilde_2 \big\rangle,
\ee
which satisfies the following key inequality.
\begin{lemma} \label{l_residual}
Given $\Delta$ as in \eqref{delta}, let $\tau_j$ be an arbitrary positive number, set $\tau_i  =  \sum_{j\in \cJ_{c,i}} \tau_j$ and $\beta_i = \frac{1}{\sum_{l\in\cP_i} \beta_l^{-1}}$. Then it holds that $\forall x_1 = (u_{1,0},u_{1,i}, z_{1,i}, s_{1,j})$ and 
$ x_2 = (u_{2,0},u_{2,i}, z_{2,i}, s_{2,j})$: 
\begin{align*}
\Delta & \ge   -\sum_{i=1}^{n-1} \Big( 
\tau_i + \frac{1}{4 \beta_{i} }  \Big)
  \big\|   (z_{1, i}  - \ztilde_{1, i} ) - (z_{2, i}  - \ztilde_{2, i} )\big\|^2 \\
    & - \sum_{j=1}^m   \Big( \frac{\|  L_{j} \|^2 } {4 \tau_{j} }
+  \frac{1}{4\nu_j} \Big)
\big\|  (s_{1,j}  -  \stilde_{1,j})  - (s_{2,j}  -  \stilde_{2,j})\big\|^2  .
\end{align*}
\end{lemma}

We denote the output of relaxation step by $x^+ = \cT_{uzs,\theta,\zeta} x$. Then it is easy to obtain the following key inequalities based on Lemma \ref{l_residual}.
\begin{proposition} \label{p_ine}
For any $x_1, x_2\in \cH^{2n-1} \times \prod_{j=1}^m \cK_j$, 
it hold that:
\begin{align*}
& \sum_{i=1}^{n-1} \frac{1}{\theta_i}  \Big( \big\|  z_{2,i} - z_{1,i} \big\|_{M_i}^2 - \big\|  z_{2,i}^+ - z_{1,i}^+\big\|_{M_i}^2 \Big)
+ \sum_{j=1}^{m} \frac{1}{\zeta_j} \Big( \big\|  s_{2,j} - s_{1,j} \big\|_{N_j}^2
- \big\|  s_{2,j}^+ - s_{1,j}^+  \big\|_{N_j}^2 \Big) \\ 
\ge &   \sum_{i=1}^{n-1}  \frac{1}{\theta_i^2}
  \big\|  ( z_{2,i} - z_{2,i}^+  ) - ( z_{1,i} - z_{1,i}^+ ) \big\|_{ (2-\theta_i) M_i - ( 2\tau_i  + \frac{1}{2 \beta_{i} } ) I }^2 \\
+ & \sum_{j=1}^{m}  \frac{1}{\zeta_j^2}
 \big\| ( s_{2,j} - s_{2,j}^+  ) - ( s_{1,j} - s_{1,j}^+ ) \big\|_{(2- \zeta_j) N_j - (\frac{\|  L_{j} \|^2 } {2 \tau_{j} } +
     \frac{1}{2 \nu_j} ) I}^2.
\end{align*}
\end{proposition}

Proposition \ref{p_ine} immediately implies the $(\frac{1}{\theta_i} M_i, \frac{1}{\zeta_j}  N_j)$-based demiclosedness of $\cI-\cT_{zs,\theta,\zeta}$, which is essential for the convergence analysis.
\begin{corollary} \label{c_demi}
Given the fixed-point operator $\cT_{zs, \theta,\zeta}$ defined by \eqref{proto_T_relax}, 
$\cI-  \cT_{zs,\theta,\zeta }$ is $( \frac{1}{\theta_i} M_i, \frac{1}{\zeta_j}  N_j)$-demiclosed,
 if 

{\rm (i)}  $\theta_i,\zeta_j \in (0,2)$,  $\forall i\in [n-1]$ and $\forall j\in [m]$;

{\rm (ii)}  $M_i \succeq \alpha_i I$ with $\alpha_i > \frac{1}{2-\theta_i} (2\tau_i+\frac{1}{2\beta_i} )$, $\forall i\in [n-1]$;  

{\rm (iii)} $N_j \succeq \sigma_j I$ with $\sigma_j > \frac{1}{2-\zeta_j} (\frac{\|  L_{j} \|^2 } {2 \tau_{j} } +  \frac{1}{2 \nu_j}  )$,   $\forall j\in [m]$;

\noindent 
where $\tau_j $, $\tau_i$ and $\beta_i$ are defined in Lemma \ref{l_residual}. 
\end{corollary}
\begin{proof}
Note that the standard definitions of non-expansiveness \cite[Definition 4.1-(ii)]{plc_book} and demiclosedness \cite[Definition 4.26]{plc_book} has been naturally extended to  $\cQ$-based concepts in \cite[Definition 2.1-(ii)]{fxue_rima} and \cite[Definition 2.2]{fxue_jota}. 
Observing Proposition \ref{p_ine} and noting $(z_i^+,s_j^+) = \cT_{zs,\theta,\zeta} (z_i,s_j)$,   $ \cT_{zs,\theta,\zeta }$ is $( \frac{1}{\theta_i} M_i, \frac{1}{\zeta_j}  N_j)$-nonexpansive under the above condition. Then, $\cI - \cT_{zs,\theta,\zeta }$ is $( \frac{1}{\theta_i} M_i, \frac{1}{\zeta_j}  N_j)$-demiclosed by \cite[Lemma 2.4]{fxue_jota}---an extension of Browder's demiclosedness principle \cite[Theorem 4.27]{plc_book}.
\end{proof}

\subsection{Convergence results}
The weak convergence of the sequence  $\big\{ (z_i^k, s_j^k) \big\}_{k\in\N}$ is stated as below.   
\begin{theorem} \label{t_zs}
Let  $\big\{ (u_0^k, u_i^k, z_i^k, s_j^k) \big\}_{k\in\N}$ be a  sequence generated by the relaxed $(u,z,s)$-scheme  \eqref{proto_T_relax}. Then, under the conditions in Corollary \ref{c_demi},  the following hold:

{\rm (i)} $\exists (z_i^\star, s_j^\star) \in \Fix \cT_{zs, \theta,\zeta}$, such that $ (z_i^k, s_j^k) \weak  (z_i^\star, s_j^\star) $, as $k\rightarrow \infty$. 

{\rm (ii)}  $ J_{A_0}^{\sum_{i_1=1}^{ |\cI_1|} M_{i_1}} 
\big( \sum_{i_1=1}^{ |\cI_1|} M_{i_1} z_{i_1}^\star -
 \sum_{j \in\cJ_{0 } } L_{j}^* s_{j}^\star  + a  \big)$ solves \eqref{p}.
\end{theorem}

Similar to \cite[Eq. (3.21)]{bredies_lift}, the desired $o(1/k)$-rate of asymptotic regularity can also be obtained. 
\begin{proposition} \label{p_rate}
Under the conditions of Corollary \ref{c_demi}, 
 $\sum_{i=1}^{n-1} \frac{1}{\theta_i}  \|z_i^{k} - z_i^{k+1} \big\|_{M_i}^2 
+ \sum_{j=1}^{m} \frac{1}{\zeta_j}  \|s_j^{k} - s_j^{k+1} \big\|_{N_j }^2$ has the pointwise rate of  $o(1/k)$. More specifically, 
$\forall k \in \N$,
{\small \[
 \sum_{i=1}^{n-1} \frac{1}{\theta_i} \big\|z_i^{k} - z_i^{k+1} \big\|_{M_i }^2 
+ \sum_{j=1}^{m}  \frac{1}{\zeta_j}\big\|s_j^{k} - s_j^{k+1} \big\|_{N_j}^2 
 \le  \frac{1}{  k \cdot (\xi-1)} \bigg( 
 \sum_{i=1}^{n-1} \frac{1}{\theta_i}  \big\|z_i^{0} - z_i^\star \big\|_{M_i}^2 
+ \sum_{j=1}^{m} \frac{1}{\zeta_j} \big\|s_j^{0} - s_j^\star \big\|_{N_j }^2 \bigg) ,
\]}%
where $\xi = \min_{  i\in [n-1],  j\in [m] } \Big\{
\frac{2}{\theta_i} \big( 1- \frac{1} {\alpha_{i}}  (\tau_i+\frac{1}{4\beta_i}) \big), 
\frac{2}{\zeta_j} \big( 1-  \frac{1} {\sigma{j}} 
(\frac{\|  L_{j} \|^2 } {4 \tau_{j} } +  \frac{1}{4 \nu_j}  ) \big)  \Big\}$.
\end{proposition}

The typical choice of $M_i=\gamma_i I$ and $N_j = \eta_j I$ deserves particular attention.  
\begin{corollary} \label{c_I}
Let  $\big\{ (u_0^k, u_i^k, z_i^k, s_j^k) \big\}_{k\in\N}$ be a  sequence generated by the relaxed $(u,z,s)$-scheme  \eqref{proto_T_relax} with  $M_i=\gamma_i I$ and $N_j = \eta_j I$. 
If $\theta_i,\zeta_j\in (0,2)$,  
 $\gamma_i >  \frac{1} {2-\theta_i} ( 2\tau_i  + \frac{1}{2 \beta_{i} } ) $,  and $\eta_j > \frac{1} {2-\zeta_j} \big( \frac{\|  L_{j} \|^2 } {2 \tau_{j} } +  \frac{1}{2 \nu_j} ) $,  $\forall i\in [n-1]$,  $\forall j\in [m]$, where  $\tau_j$, $\tau_i$ and $\beta_i$ are defined in Lemma \ref{l_residual},  then the following holds:
 
{\rm (i) [weak convergence]} $\exists (z_i^\star, s_j^\star) \in \Fix \cT_{zs, \theta,\zeta}$, such that $ (z_i^k, s_j^k) \weak  (z_i^\star, s_j^\star)$, as $k\rightarrow \infty$;

{\rm (ii) [fixed-point encoding]} $ J_{A_0}^{ (\sum_{i_1=1}^{ |\cI_1|} \gamma_{i_1}) I } 
\Big( \sum_{i_1=1}^{ |\cI_1|} \gamma_{i_1} z_{i_1}^\star -
 \sum_{j \in\cJ_{0 } } L_{j}^* s_{j}^\star  + a  \Big)$ solves \eqref{p};

{\rm (iii) [asymptotic regularity]} $ \sum_{i=1}^{n-1} \frac{\gamma_i}{\theta_i} \big\|z_i^{k} - z_i^{k+1} \big\|^2 
+ \sum_{j=1}^{m}  \frac{\eta_j} {\zeta_j}\big\|s_j^{k} - s_j^{k+1} \big\|^2 $ has the pointwise rate of  $o(1/k)$. More specifically, $\forall k \in \N$,
{\small \[
 \sum_{i=1}^{n-1} \frac{\gamma_i}{\theta_i} \big\|z_i^{k} - z_i^{k+1} \big\|^2 
+ \sum_{j=1}^{m}  \frac{\eta_j} {\zeta_j}\big\|s_j^{k} - s_j^{k+1} \big\|^2 
 \le  \frac{1}{  k \cdot (\xi-1) } \bigg( 
 \sum_{i=1}^{n-1} \frac{\gamma_i}{\theta_i}  \big\|z_i^{0} - z_i^\star \big\|^2 
+ \sum_{j=1}^{m} \frac{\eta_j}{\zeta_j} \big\|s_j^{0} - s_j^\star \big\|^2 \bigg) ,
\]}%
where   $\xi = \min_{  i\in [n-1],  j\in [m] } \Big\{
\frac{2}{\theta_i} \big( 1- \frac{1}{\gamma_i} (\tau_i+\frac{1}{4\beta_i}) \big) , 
\frac{2}{\zeta_j} \big( 1- \frac{1} {\eta_j} (  \frac{\|  L_{j} \|^2 } {4 \tau_{j} } +  \frac{1}{4 \nu_j} ) \big)  \Big\}$.   
\end{corollary}
\begin{proof}
(i)--(ii): by Theorem \ref{t_zs}.

(iii) is from Proposition \ref{p_rate}.
\end{proof}

\section{Applications}
We first show that the particular instances of our scheme extend the (forward-)DRS and parallel Chambolle--Pock. Particularly, Theorem \ref{t_u} proves the weak convergence of solution trajectories  $\big\{ (u_0^k, u_i^k ) \big\}_{k\in\N}$ for the case of $m=0$ and $p=0$. When applying to convex minimization, we propose a primal-dual gap function, which achieves $O(1/k)$ ergodic rate by Proposition \ref{p_gap}. 

\subsection{Distributed extensions of the forward-DRS}
If $m=0$ in the problem \eqref{p}, the tree in Fig. \ref{fig_whole} does not involve dual assignments and corrections. Then the relaxed scheme \eqref{proto_T_relax} is simplified to: 
{\small
\be   \label{pure}
\left\lfloor \begin{aligned}
u_{ 0 }^{k+1}  & = J_{ A_{0 } }^{S_0 } \bigg(  \sum_{i_{1} =1}^{ |\cI_{1} | } M_{i_1 }  z_{i_1 }^{k} +a \bigg)  ; \\
u_{ i_{q} }^{k+1}  & = J_{ A_{i_{q} } }^{S_{ i_{q} } } 
\bigg(  M_{ i_{q} } \big( 2 u_{ i_{q-1} }^{k+1} -z_{ i_{q} }^{k}   \big)  +  \sum_{i_{q+1} =1}^{ |\cI_{q+1,i_1,\ldots, i_q} | }
M_{ i_{q+1} }  z_{ i_{q+1} }^{k} 
- \sum_{l\in \cP_{i_1,\ldots,i_q}} C_l u_{i_1,\ldots,i_{q-1}}^{k+1} \bigg) ; \\
z_{ i_{q}}^{k+1} & =  z_{ i_{q}}^{k} + \theta_{i_q}
 \big( u_{ i_{q}}^{k+1}  - u_{ i_{q-1} }^{k+1} \big), \quad 
\forall i_q \in \big[\big| \cI_{q,i_1,\ldots, i_{q-1} } \big| \big] , \\
\end{aligned} \right. 
\ee}%
where  $S_0=\sum_{i_{1} =1}^{ |\cI_{1} | } M_{i_1 } $,  $S_{i_q} =  M_{i_{q} } + \sum_{i_{q+1} =1}^{ |\cI_{q+1,i_1,\ldots, i_q} | } M_{ i_{q+1} }$.  This is a distributed extension of  the forward-DRS \cite{vu_2020,arias_half,bredies_ppa} and Davis--Yin \cite{ywt_2017} with flexible level-synchronous distributed computing (including parallel-up and sequential).  The convergence result is given below.
\begin{corollary}  
Let  $\big\{ (u_0^k, u_i^k, z_i^k ) \big\}_{k\in\N}$ be a  sequence generated by \eqref{pure}. If $\theta_i  \in (0,2)$, 
 $M_i \succ  \frac{1} {2\beta_i (2-\theta_i)}  I$, $\forall i\in [n-1]$,  the following hold:

{\rm (i)} $\exists (z_i^\star) \in \Fix \cT_{z, \theta}$, such that $  z_i^k  \weak  z_i^\star$, $\forall i\in [n-1]$, as $k\rightarrow \infty$. 

{\rm (ii)}  $ J_{A_0}^{\sum_{i_1=1}^{ |\cI_1|} M_{i_1}} 
\big( \sum_{i_1=1}^{ |\cI_1|} M_{i_1} z_{i_1}^\star  + a  \big)$ solves \eqref{p} with $m=0$.

{\rm (iii)}  $\sum_{i=1}^{n-1} \frac{1}{\theta_i}  \|z_i^{k} - z_i^{k+1} \big\|_{M_i}^2$ has the pointwise rate of  $o(1/k)$. More specifically,
\[
 \sum_{i=1}^{n-1} \frac{1}{\theta_i} \big\|z_i^{k} - z_i^{k+1} \big\|_{M_i }^2 
 \le  \frac{1}{  k \cdot (\xi-1)} 
 \sum_{i=1}^{n-1} \frac{1}{\theta_i}  \big\|z_i^{0} - z_i^\star \big\|_{M_i}^2 , \quad \forall k \in \N,
\]
where $\xi = \min_{  i\in [n-1] } \Big\{
\frac{2}{\theta_i} \big( 1- \frac{1} {4\alpha_{i}  \beta_i} \big)  \Big\}$, $\alpha_i$ is such that $M_i \succeq \alpha_i I$, $\forall i\in [n-1]$.
\end{corollary}
\begin{proof}
Clear from Theorem \ref{t_zs} and Proposition \ref{p_rate}. 
\end{proof}

\subsection{Pure resolvent splitting}
We now consider the case of $m=0$ and $p=0$, where our scheme becomes the so-called {\it pure resolvent splitting} \cite[Sect. 3.1]{full} or distributed DRS \cite{campoy_product,bredies_lift}:
{\small
\be   \label{pure_drs}
\left\lfloor \begin{aligned}
u_{ 0 }^{k+1}  & = J_{ A_{0 } }^{S_0 } \bigg(  \sum_{i_{1} =1}^{ |\cI_{1} | } M_{i_1 }  z_{i_1 }^{k} +a \bigg)  ; \\
u_{ i_{q} }^{k+1}  & = J_{ A_{i_{q} } }^{S_{ i_{q} } } 
\bigg(  M_{ i_{q} } \big( 2 u_{ i_{q-1} }^{k+1} -z_{ i_{q} }^{k}   \big)  +  \sum_{i_{q+1} =1}^{ |\cI_{q+1,i_1,\ldots, i_q} | }
M_{ i_{q+1} }  z_{ i_{q+1} }^{k}  \bigg) 
, \quad \forall i_q\in \big[\big| \cI_{q,i_1,\ldots, i_{q-1} } 
\big| \big] ; \\
z_{ i_{q}}^{k+1} & =  z_{ i_{q}}^{k} + \theta_{i_q}
 \big( u_{ i_{q}}^{k+1}  - u_{ i_{q-1} }^{k+1} \big), \quad 
\forall i_q \in \big[\big| \cI_{q,i_1,\ldots, i_{q-1} } \big| \big] , \\
\end{aligned} \right. 
\ee}%

The apparent and reduced fixed-point operators become $\cT_{uz,\theta}$ and $\cT_{z,\theta}$, respectively, where $\cT_{uz,\theta} = \cI - \begin{bmatrix}
1 & 0&0  \\   0& 1&0 \\
0&0 & \theta_i  \end{bmatrix} (\cI -\cT_{uz})$ for the input $x = (u_0,u_i,z_i) \in \cH^{2n-1}$, and 
$\cT_{z,\theta} = \cI -  \theta (\cI -\cT_{z})$ for the input $( z_i) \in \cH^{n-1}$.
In particular, the centralized or decontralized DRS can be further obtained, when \eqref{pure_drs} is equipped with the star-shaped  (Fig. \ref{fig_center}) or the fully sequential graph (Fig. \ref{fig_decenter}), respectively.  Corollary \ref{c_drs} gives the basic convergence results.
\begin{corollary}   \label{c_drs}
Let  $\big\{ (u_0^k, u_i^k, z_i^k ) \big\}_{k\in\N}$ be a  sequence generated by \eqref{pure_drs}. If $\theta_i  \in (0,2)$, the following hold:

{\rm (i)} $\exists (z_i^\star) \in \Fix \cT_{z, \theta}$, such that $  z_i^k  \weak  z_i^\star$, $\forall i\in [n-1]$, as $k\rightarrow \infty$. 

{\rm (ii)}  $ J_{A_0}^{\sum_{i_1=1}^{ |\cI_1|} M_{i_1}} 
\big( \sum_{i_1=1}^{ |\cI_1|} M_{i_1} z_{i_1}^\star  + a  \big)$ solves \eqref{p} with $m=0$ and $p=0$.

{\rm (iii)}  $\sum_{i=1}^{n-1} \frac{1}{\theta_i}  \|z_i^{k} - z_i^{k+1} \big\|_{M_i}^2$ has the pointwise rate of  $o(1/k)$. More specifically,
\[
 \sum_{i=1}^{n-1} \frac{1}{\theta_i} \big\|z_i^{k} - z_i^{k+1} \big\|_{M_i }^2 
 \le  \frac{\theta_{\max} }{  k \cdot (2- \theta_{\max})} 
 \sum_{i=1}^{n-1} \frac{1}{\theta_i}  \big\|z_i^{0} - z_i^\star \big\|_{M_i}^2 , \quad \forall k \in \N,
\]
where $\theta_{\max} = \max_{  i\in [n-1] } \theta_i $.
\end{corollary}
\begin{proof}
Clear from Theorem \ref{t_zs} and Proposition \ref{p_rate}. 
\end{proof}

\vskip.2cm
More remarkably, Theorem \ref{t_u} shows the weak convergence of the primal variables $\big\{ (u_0^k, u_i^k) \big\}_{k\in\N}$ (i.e., the solution trajectories) to  $u^\star$---a solution to \eqref{p} for this case. This extends the classic result  of DRS \cite{svaiter} to distributed versions.
To show this, we first present a simple fact of the fixed-point operator $\cT_{uz}$. 
\begin{fact} \label{f_ppm}
Let $\cT_{uz}$ be the apparent fixed-point operator defined by \eqref{pure_drs} with $\theta_{i_q} = 1$. Then  the following holds:

{\rm (i)} $\Fix \cT_{uz} =\Fix \cT_{uz,\theta} $;

{\rm (ii)} $ \cT_{uz} = (\cA+\cQ)^{-1} \cQ $, where $\cA$ and $\cQ$ are given in Lemma \ref{l_ppm};

{\rm (iii)} $\zer \cA = \Fix\cT_{uz}$. 
\end{fact}
\begin{proof}
(i) is from the definition of $\cT_{uz,\theta}$;

(ii)-(iii) is clear by Lemma \ref{l_ppm}, noting that the residual term $r^k$ vanishes in this special case.
\end{proof}

The following basic fact regarding the firm nonexpansiveness of metric resolvent slightly extends  \cite[Lemma 2.4]{plc}.
\begin{fact} \label{f_res}
Given  $a_i = J_{A}^M (b_i)$ for $i=1,2$, then $\|a_1-a_2\|_M^2 \le \|b_1-b_2\|_{M^{-1}}^2 -  \|(a_1-a_2)-M^{-1} (b_1-b_2)\|_M^2$.
\end{fact}
\begin{proof}
According to the definition $J_A^M: = J_{M^{-1}A} \circ M^{-1}$ (cf. Sect. \ref{sec_notation}), the inclusion form is
$0\in Aa_i +Ma_i -b_i$. Applying the monotonicity of $A$ yields $0\le \langle a_1-a_2 | Aa_1-Aa_2\rangle = -\|a_1-a_2\|_M^2 + \langle a_1-a_2|b_1-b_2\rangle$. Simple algebraic arrangement yields $\|a_1-a_2\|_M^2 \le \|b_1-b_2\|_{M^{-1}}^2 - \|(a_1-a_2)-M^{-1} (b_1-b_2)\|_M^2$.
\end{proof}

The following lemma shows that $ \cT_{uz,\theta} $ satisfies the Lipschitz-like condition (cf. \cite[Assumption 7-(iii)]{fxue_jota}), which will be used in Theorem \ref{t_u}.
\begin{lemma}  \label{l_rho}
Given $ \cT_{uz,\theta} $ defined in \eqref{pure_drs}, there exists a constant $\rho > 0$, such that 
$\big\| \cT_{uz,\theta} (u_{1,0}, u_{1,i}, z_{1,i} ) - 
\cT_{uz,\theta} (u_{2,0}, u_{2,i}, z_{2,i} )  \big\|^2 
\le \rho \cdot \sum_{i=1}^{n-1} 
\big\| z_{1,i} - z_{2,i} \big\|_{M_i}^2$, $\forall (u_{1,0}, u_{1,i}, z_{1,i} ), (u_{2,0}, u_{2,i}, z_{2,i} ) \in \cH\times \cH^{n-1} \times \cH^{n-1}$.
\end{lemma}

Finally, the weak convergence of solution sequence $\big\{ (u_0^k,u_i^k) \big\}_{k\in\N}$ can be proved.
\begin{theorem} \label{t_u}
Let  $\big\{ (u_0^k, u_i^k, z_i^k) \big\}_{k\in\N}$ be a  sequence generated by the scheme  \eqref{pure_drs}. If $\theta_i \in (0,2)$, then $\exists (z_i^\star)\in \Fix \cT_{z}$, such that $z_i^k \weak z_i^\star$, $u_0^k \weak u^\star$ and $u_i^k \weak u^\star$, $\forall i\in [n-1]$, where $u^\star =  J_{A_0}^{\sum_{i_1=1}^{ |\cI_1|} M_{i_1}} 
\big( \sum_{i_1=1}^{ |\cI_1|} M_{i_1} z_{i_1}^\star +a\big)$ solves \eqref{p} with $m=0$ and $p=0$.
\end{theorem}
\begin{proof}
First, by Corollary \ref{c_existence}, Theorem \ref{t_zs} and Proposition \ref{p_ine}, we have $(z_i^k) \weak (z_i^\star) \in \Fix \cT_z$ and $z_i^{k+1} - z_i^{k} \rightarrow 0$ ($\forall i\in [n-1]$),  as $k\rightarrow +\infty$.
Combining Proposition \ref{p_ine} with Lemma \ref{l_rho}, the sequence  $\big\{ (u_0^k, u_i^k, z_i^k ) \big\}_{k\in\N}$ is bounded.
Take a subsequence $\big\{ (u_0^{k_l}, u_i^{k_l}, z_i^{k_l} ) \big\}_{k\in\N}$, such that $(u_0^{k_l}, u_i^{k_l}, z_i^{k_l} ) \weak (u_0^\star , u_i^\star , z_i^\star )$ as $k_l\rightarrow +\infty$, and particularly, the relaxation step yields $\cT_z (z_i^{k_l}) - (z_i^{k_l}) = \big[ \theta_i^{-1} \big] ( z_i^{k_l+1} -  z_i^{k_l} )  \rightarrow 0$, $\forall i\in [n-1]$.

On the other hand, applying Lemma \ref{l_rho} to $x_1 = x^k$ and $x_2 = x^{k-1}$, then $\exists \rho >0$, such that $\|x^{k+1} - x^k\|^2 \le \rho \cdot \sum_{i=1}^{n-1} \|z_i^k - z_i^{k-1} \|_{M_i}^2$. Then we have $x^{k+1} - x^k \rightarrow 0$, and hence, by the relaxation step, $\cT_{uz} x^k - x^k = \begin{bmatrix}
1 &0&0 \\0 &1&0 \\0 &0& \theta_i^{-1} 
\end{bmatrix} ( x^{k+1} - x^k)
 \rightarrow 0$, as $k\rightarrow +\infty$. Therefore, for the above subsequence $\big\{ x^{k_l} \big\}_{k\in\N}$, we have $\cT_{uz} x^{k_l} \weak (u_0^\star , u_i^\star , z_i^\star )$.

Based on Lemma \ref{l_ppm}, the prediction step of \eqref{pure_drs} (corresponding to $\theta_{i_q} = 1$) can be recast as $\cA \cT_{uz} x^{k_l} \owns \cQ (x^{k_l} -\cT_{uz} x^{k_l} )$\footnote{Notice that the residual term $r^k$ vanishes in this case.}. Since  $\cT_{uz} x^{k_l} \weak (u_0^\star, u_i^\star, z_i^\star )$  and  $ \cQ (x^{k_l} -\cT_{uz} x^{k_l} ) 
= \big( 0,0, M_i (z_i^{k_l} -\cT_z z_i^{k_l} ) \big) \rightarrow (0,0,0 )$, it then follows from the closedness of $\gra\cA$ in  $\cH^\text{weak} \times \cH^\text{strong} $ (due to maximaly monotonicity of $\cA$ by \cite[Proposition 20.38-(ii)]{plc_book}) that $(u_0^\star, u_i^\star, z_i^\star ) \in \zer\cA = \Fix\cT_{uz} = \Fix\cT_{uz,\theta} $ by Fact \ref{f_ppm}, Since the subsequence  $\big\{ (u_0^{k_l}, u_i^{k_l}, z_i^{k_l} ) \big\}_{k\in\N}$ is arbitrary, this indicates that every weak sequential cluster point of $\{x^k\}_{k\in\N}$---$(u_0^\star, u_i^\star, z_i^\star )$---lies in $\Fix \cT_{uz}  = \Fix \cT_{uz,\theta} $. This further implies that $(u_0^\star, u_i^\star, z_i^\star ) = \cT_{uz} (u_0^\star, u_i^\star, z_i^\star ) = (\cA+\cQ)^{-1} \cQ (u_0^\star, u_i^\star, z_i^\star ) =  (\cA+\cQ)^{-1}  (0, 0, M_i z_i^\star ) $. Since $z_i^\star$ is unique by Theorem \ref{t_zs}, $u_0^\star$ and $u_i^\star$ are also unique.
Finally, the weak convergence of  $\big\{ (u_0^{k}, u_i^{k}, z_i^{k} ) \big\}_{k\in\N}$, is obtained by \cite[Lemma 2.46]{plc_book}. Moreover, $u_0^\star = u_i^\star  = J_{A_0}^{\sum_{i_1=1}^{ |\cI_1|} M_{i_1}} \Big( \sum_{i_1=1}^{ |\cI_1|} M_{i_1} z_{i_1}^\star +a \Big)$, $\forall i\in [n-1]$, by Proposition \ref{p_encoding}-(i).
\end{proof}

\subsection{Recovery of parallel Chambolle--Pock}
We now would like to show that the Chambolle--Pock algorithm and its parallel version can also be exactly recovered from our scheme, for the case $n=1$, $m\in\N$ and $p=0$ of the problem \eqref{p_simple}:  
\be \label{min_eg}
\min_{u\in\cH} \Big(f +\sum_{j=1}^m (g_j\circ L_j) \Big) (u),
\ee
where $f$ and $g_j$ with $j\in [m]$ are proper, convex, lower semi-continuous and prox-friendly functions.
\begin{proposition} \label{p_cp}
The scheme \eqref{proto}, when being applied to the problem \eqref{min_eg} with  $M_1 = \frac{1}{\tau} I$ and $N_j = \frac{1}{\sigma_j} I$, is simplified as the parallel Chambolle--Pock \cite[Eq. (202)]{condat_tour}:
\be \label{ay}
\left\lfloor \begin{aligned}
u^{k+1}  & =\prox_{\tau f} \Big( u^{k}
-\tau  \sum_{j=1  }^m L_{j}^*  (2 s_{j}^{k } - s_j^{k-1} ) \Big)
  ; \\
s_{j }^{k+1} & = \prox_{  \sigma_j g_j^* } \big(  s_{j }^k+
\sigma_j   L_{j  } u^{k+1} \big) ,\quad \forall j\in [m] ,
\end{aligned} \right. 
\ee
Particularly, if $m=1$,  \eqref{ay} becomes the standard Chambolle--Pock \cite[Algorithm 1]{cp_2011}.
\end{proposition}
\begin{proof}
Let $a=0$, $A_0 = \partial f$ and $B_j = \partial g_j$, $D_j = \cN_{\{0\}}$, $b_j=0$, $\forall j\in [m]$.
Considering the primitive $(u,w,s)$-scheme \eqref{proto}, set a {\it virtual} primal node $A_1 \equiv 0$ as a child of $A_0$ and let  $M_1 = \frac{1}{\tau} I$ and $N_j = \frac{1}{\sigma_j} I$. Since there are only two primal nodes, one has to set $\cJ_0=[m]$. Then \eqref{proto} becomes:
\be \label{rw}
\left\lfloor \begin{aligned}
0 & \in   \partial f ( u_{ 0 }^{k+1})   + \frac{1}{\tau} 
  (u_{0 }^{k+1}  - u_{ 1 }^{k}  +  w_{1 }^k )  
+  \sum_{j =1 }^m L_{j}^* s_{j}^{k}  ; \\
0 & \in   \partial g_j^* ( s_{j } ^{k+1} ) + \frac{1}{\sigma_j} (s_{j }^{k+1} - s_{j }^{k}) - L_{j  } u_0^{k+1} , \quad  \forall j  \in [m]  ; \\
0 & =  w_{1}^{k+1} -w_{ 1 }^{k} + u_{ 1 }^{k}  - u_{0}^{k+1}; \\
0 &  \in  A_{1 } u_{  1 }^{k+1}   + \frac{1}{\tau} 
  (u_{ 1 }^{k+1}  - u_{0 }^{k+1}  -   w_{ 1 }^{k+1} )   +  
   \sum_{j=1  }^m L_{j}^*  (s_{j}^{k+1} - s_j^k ) .
\end{aligned} \right. 
\ee
Since $A_1\equiv 0$, the last line of \eqref{rw} becomes 
$ u_{ 1 }^{k } -   w_{ 1 }^{k }  =   u_{0 }^{k } - \tau  
   \sum_{j=1  }^m L_{j}^*  (s_{j}^{k } - s_j^{k-1} )$.
Substituting into the first inclusion of \eqref{rw} yields
$0  \in   \partial f ( u_{ 0 }^{k+1})   + \frac{1}{\tau} 
  (u_{0 }^{k+1}  - u_{ 0 }^{k} )  
+   \sum_{j=1  }^m L_{j}^*  (2 s_{j}^{k } - s_j^{k-1} )$, which completes the proof. 
\end{proof}

As summary, Table \ref{table_drs} lists the special instances of our \eqref{proto_T_relax}. Note that the Chambolle--Pock \cite[Algorithm 1]{cp_2011} cannot be recovered from  \cite[Algorithm 1]{fjaa_pds}.
\begin{table}[h!]
  \centering
   \caption{Recovery of existing methods from our  scheme \eqref{proto_T_relax}.}\vspace{-1em}
   \resizebox{1.0\columnwidth}{!} {
 \begin{tabular}{|l||l|l|} 
    \hline
     methods &  graph topology & instance of our \eqref{proto_T_relax}  \\
    \hline\hline
    DRS \cite{lions} & 
     \makecell[l]{ unique tree  (Fig. \ref{fig_center}/\ref{fig_decenter}) }  
     & $M_1 = \frac{1}{\gamma} I$, $\cP_1 = \varnothing$,
     $\cJ_0 = \varnothing$  \\
\hline
 \tabincell{l}{ Davis--Yin \\  \cite[Algorithm 1]{ywt_2017} }     
     &    \makecell[l]{unique tree (Fig. \ref{fig_center}/\ref{fig_decenter}) \\ loading cocoercive on child}  & 
       $M_1 = \frac{1}{\gamma} I$, $\cP_1 = \{1\}$, 
        $\cJ_0 = \varnothing$   \\
\hline
    \cite[Theorem 5.1]{campoy_product} & Fig. \ref{fig_center} &  \makecell[l]{ $M_i = \frac{1}{\gamma} I$,  $\cP_i = \varnothing$, \\   $\cJ_i = \varnothing$ ,
    $\forall i\in [n-1]$  } \\
\hline
   \makecell[l]{  \cite[Eq.(3.14)]{bredies_ppa} \\ 
    \cite[Example 4.5]{fjaa_graph} }    & Fig. \ref{fig_center} & \multirow{2}{*} [-1ex] { \makecell[l]{  $M_i = \frac{1}{\gamma} I$,  $\cP_i = \{ i \}$, \\   $\cJ_i = \varnothing$,    $\forall i\in [n-1]$  }  } \\
\cline{1-2}
   \makecell[l]{  \cite[Eq.(3.15)]{bredies_ppa} \\ 
    \cite[Example 4.4]{fjaa_graph} }   & Fig. \ref{fig_decenter} &  \\
\hline
    \cite[Eq.(197)]{condat_tour}  & Fig. \ref{fig_center} &
    \makecell[l]{  $M_i = \frac{1}{\tau} w_i I$, where $\sum_{i=1}^{n-1} w_i = 1$ \\ $\cP_i=\varnothing$, $\cJ_i = \varnothing$,  $\forall i\in [n-1]$} \\
\hline 
    \cite[Algorithm 1]{cp_2011} & \multirow{2}{*}{ 
    \makecell[l]{ Fig. \ref{fig_center}/\ref{fig_decenter} consisting of  \\ two primal nodes only, \\ second node is virtual   } }  
    & $M_1 = \frac{1}{\tau}   I$, 
    $N_1 = \frac{1}{\sigma} I$,   $\cJ_0 = \{1\}$\\
\cline{1-1} \cline{3-3}
   \makecell[l]{ parallel Chambolle--Pock \\  \cite[Eq.(202)]{condat_tour} }  &  &     \makecell[l]{  $M_1 = \frac{1}{\tau}   I$,      $\cJ_0 = [m]$ \\
    $N_j = \frac{1}{\sigma_j} I$,    $\forall j\in [m]$ } \\
\hline
    \end{tabular}
   }
    \vskip 0.5em
\label{table_drs}
\vspace*{-0.3cm}
\end{table}

\subsection{Convex minimization problem}
We consider the following problem:
\be \label{primal}
\inf_{u\in \cH} \sum_{i=0}^{n-1} f_i (u) + 
\sum_{j=1}^m (g_j\square d_j) (L_j u-b_j) +\sum_{l=1}^p
h_l(u) -\langle u|a\rangle,
\ee
where $f_i:\cH\mapsto \R\mcup \{+\infty\}$ and 
$g_j: \cK_j \mapsto \R\mcup \{+\infty\}$ are proper, lower semi-continuous, convex and prox-friendly functions,  $h_l:\cH\mapsto \R$ is convex and differentiable with $\beta_l^{-1}$-Lipschitz gradient. The symbol $\square$ denotes the infimal convolution. The function $d_j$ is $\nu_j^{-1}$-strongly convex, which implies that $d_j^*$ is Fr\'{e}chet differentiable, and $\nabla d_j^*$ is $\nu_j$-Lipschitz, and thus,  $\nu_j$-cocoercive by Baillon--Haddad theorem \cite[Corollary 18.17]{plc_book}.

Based on \cite[Problem 2.1]{bot_jmiv_2014},  the dual problem to \eqref{primal} is given as
\vspace*{-.3cm}
\[
\sup_{\substack{ s_1,\ldots,s_j \in \\ \cK_1\times \cK_m} } 
\bigg\{ - \bigg( \Big(\sum_{i=0}^{n-1} f_i \Big)^* \square
\Big(\sum_{l=1}^p h_l \Big)^* \bigg) 
\Big( a-\sum_{j=1}^m L_j^* s_j \Big)
- \sum_{j=1}^m \Big( g_j^*(s_j) +  d_j^* (s_j) + \langle s_j | b_j\rangle \Big)  \bigg\}.
\vspace*{-.3cm}
\]
Assume some qualification condition (see, for instance, \cite[Remark 2.1]{plc_dual_2010}) is fulfilled, such that the optimal primal-dual solution pair $(u^\star, s_j^\star)_{j\in [m]}$ exists, and strong duality holds. Most existing algorithms can only handle the cases of $n=1$ \cite{bot_jmiv_2014,vu_2013}, $m=0$ \cite{bredies_lift,tam_2026,morin_frugal} or $d_j = \iota_{ \{0\} }$ \cite{fjaa_pds}.
Then, applying the Fermat's rule \cite[Theorem 16.3]{plc_book} to \eqref{primal} yields \eqref{p} with the following correspondences: $A_i = \partial f_i$, $B_j^{-1} = \partial g_j^*$, $D_j^{-1} = \nabla d_j^*$ and $C_l = \nabla h_l$. This can be addressed by our relaxed scheme \eqref{proto_T_relax}. Thus, our Theorem \ref{t_zs} and Proposition \ref{p_rate} apply to this setting. 

\vskip.2cm
Now, we particularly consider the special case of $m=0$ and $p=0$ in \eqref{p}, i.e., 
\be \label{p_inf}
\inf_{u\in \cH}  \underbrace{ f_0(u) -\langle u|a\rangle}
_{:=\tilde{f}_0 (u)}
+ \sum_{i=1}^{n-1} f_i(u).
\ee
We apply the $(u,w)$-prototype \eqref{proto}\footnote{Recall that \eqref{proto} is a special case of \eqref{proto_T_relax} with $\theta_i=1$ and $\zeta_j = 1$, $\forall i\in [n-1]$, $\forall j\in [m]$.} with $m=0$ and $p=0$, and investigate the properties of the primal-dual gap, following the idea of \cite[Sect. 3]{fxue_optl}. 

Similar to Lemma \ref{l_ppm}, we first reformulate \eqref{proto} with $m=0$ and $p=0$ in a product Hilbert space. 
\begin{lemma} \label{l_ppm_w}
The $(u,w)$-prototype \eqref{proto} with $m=0$ and $p=0$ can be reformulated as $ 0\in \cA x^{k+1} + \cQ (x^{k+1} - x^k)$ in a product Hilbert space, where $x$, $\cA$ and $\cQ$ are specified below:

$\bullet$   $x = (u_0,u_i,w_i)_{i\in [n-1], j\in [m]} \in \cH\times  \cH^{n-1} \times \cH^{n-1}$.

$\bullet$ The upper-triangular entries of block operator $\cA$ (including the diagonal) are defined as follows (cf. Remarks \ref{rmk_notation} and \ref{rmk_notation2}):
$\cA_{u_{i_{q-1}}\rightarrow w_{i_q}} = M_{i_q}$, $i_q\in \cI_{q,i_1,\ldots,i_{q-1}}$, $\forall q\ge 1$;  $\cA_{u_{i_{q}}\rightarrow w_{i_q}} = - M_{i_q}$, $\forall q\ge 1$; 
$\cA_{u_{0}\rightarrow u_{0}}: \cH\mapsto 2^\cH: x\mapsto A_{0}x -a$,
$\cA_{u_{i_q}\rightarrow u_{i_q}} = A_{i_q}$, $\forall q\ge 1$.
The remaining entries are determined by skew-symmetry where applicable; all other unspecified entries are zero.

$\bullet$ The block preconditioner $\cQ$ is self-adjoint:
 $\cQ_{u_{i_{q}}\rightarrow u_{i_q}} = M_{i_q}$,  $\cQ_{w_{i_{q}}\rightarrow w_{i_q}} = M_{i_q}$, and $\cQ_{u_{i_{q}}\rightarrow w_{i_q}} = - M_{i_q}$,  $\cQ_{w_{i_{q}}\rightarrow u_{i_q}} = - M_{i_q}$, $\forall q\ge 1$.  All other unspecified entries are zero.
\end{lemma}
\begin{proof}
By variable substitutions, the $u_{i_1,\ldots,i_q}$-step of \eqref{proto}  with $m=0$ and $p=0$ becomes
\[
0  \in  A_{ i_{q} } u_{ i_{q} }^{k+1}   + 
M_{ i_{q} } (u_{ i_{q} }^{k+1} - u_{ i_{q} }^{k} )   - M_{ i_{q} }
(2  w_{i_{q} }^{k+1}-   w_{ i_{q} }^{k} ) 
 +  \sum_{i_{q+1} =1}^{ |\cI_{q+1,i_1,\ldots, i_q} | }
M_{ i_{q+1} }  w_{  i_{q+1} }^{k+1}  . 
\]
Then the result follows.
\end{proof}

It has been mentioned in Sect. \ref{sec_dynamical} that $w_{i_q}$ acts as the Lagrangian multiplier for the equality constraint of $u_{i_{q-1}} = u_{i_{q}}$. The standard Lagrangian formulation becomes:
\be  \label{lag}
\cL(u_0,u_i,w_i) =   f_0(u_0) - \langle u_0|a\rangle +\sum_{i=1}^{n-1} f_i (u_i)
+\sum_{q\ge 1} \sum_{i_q\in \cI_{q,i_1,\ldots,i_q-1}}
\langle  M_{i_q} w_{i_q} |  u_{i_{q-1}} - u_{i_{q}} \rangle.
\ee

Notice that the proposed \eqref{lag} is tree-dependent: distinct parent-child relations and dual assignments lead to different Lagrangian formulations.
In addition, the Lagrangian multiplier $w_{i_q}$ is rescaled by the associated metric $M_{i_q}$. The following results show that the our proposed primal-dual gap \eqref{gap_mix} constructed by \eqref{lag} enjoys the ergodic rate of $O(1/k)$.

To this end, we define a quantity as a function of arbitrary $x,x'\in \cH^{2n-1}$:
\be \label{pi}
\Pi (x,x')  = \cL(u_0,u_i,w'_i) - \cL(u'_0, u'_i, w_i),
\ee

\begin{lemma} \label{l_pi}
Under the conditions in Corollary \ref{c_demi},  the following hold $\forall x\in \cH^{2n-1}$:

{\rm (i)} $\Pi (x^{k+1},x)  \le  \big\langle x^{k+1} - x^{k} \big| x  - x^{k+1}  \big\rangle_\cQ $, where $x$ and $\cQ$ are given in Lemma \ref{l_ppm_w};

{\rm (ii)} $\Pi \big( \frac{1}{K} \sum_{k=1}^K x^{k}, x \big)  \le  \frac{1}{2K}    \sum_{i=1}^{n-1}  \big\| 
(  u_{i}^{0} - w_{i}^{0} ) - (u -  w_{i}) \big\|_{M_i  }^2 $.
\end{lemma}

For the given sets $E_1 \subset \cH$, $E_2 \subset \cH^{n-1}$ and $E_3 \subset \cH^{n-1}$, the {\it primal-dual gap} function restricted to $E_1 \times E_2 \times E_3$ is defined as
\be  \label{gap_mix}
\Psi_{E_1 \times E_2 \times E_3 } (x) = \sup_{(w'_i) \in E_3} \cL(u_0,u_i,w'_i) -  \inf_{(u'_0,u'_i) \in E_1\times E_2 } \cL(u'_0,u'_i,w_i),
\ee 
which has the upper bound in an ergodic sense:
\begin{proposition} \label{p_gap}
Under the conditions of Corollary \ref{c_demi}, if the set    $ E_1 \times E_2 \times E_3 $ is bounded, the primal-dual gap defined as \eqref{gap_mix} has the upper bound:
\[
\Psi_{E_1 \times E_2 \times E_3 }   \bigg( \frac{1}{K}  \sum_{k=1}^{K} x^{k}  \bigg)  \le \frac{1}{2K}
\sup_{x \in E_1 \times E_2 \times E_3 } 
\big\| x^0 - x \big\|_\cQ^2.
\] 
Furthermore, $\Psi_{E_1 \times E_2 \times E_3} ( \frac{1}{K}  \sum_{k=1}^{K} x^k  ) \ge 0$, if the set $E_1 \times E_2 \times E_3 $ contains a saddle point $x^\star = (u_0^\star, u_i^\star,  w_i^\star )$, such that $(u_0^\star, u_i^\star,u_i^\star - w_i^\star ) \in \Fix \cT_{uz}$.
\end{proposition}
\begin{proof}
Refer to the proofs of \cite[Theorem 2 and Corollary 2]{fxue_optl}.
\end{proof}

\section{Further discussions}
This work presented a general tree graph with arbitrary dual assignments and cocoercive loading, which automatically leads to unconditionally convergent and frugal splitting algorithms with minimal lifting. There are several future directions of both theoretical and practical interest.

First, it is known from basic linear algebra that a  matrix can be visualized as a directed graph, where the $(i,j)$-th entry represents the edge weight from node $i$ to $j$. From this viewpoint, our reformulation technique proposed in Lemma \ref{l_ppm} is inherently tied to the tree graph depicted in Fig. \ref{fig_whole}: (1) the maximally monotone operator $\cA$ is closely related to the graph Laplacian, adjacency matrix and degree matrix
 \cite{tam_2026,bredies_lift} of the tree graph in Fig. \ref{fig_whole}; (2) The active variables  and its related lifting number can be easily recognized by the range space and rank of the degenerate preconditioner $\cQ$.
This intriguing connection deserves deep investigation.
Furthermore, this reformulation technique can be developed as a unified framework for analyzing most existing splitting methods.

Second, Theorem \ref{t_zs} indicates that the convergence condition depends heavily on the tree structure, dual assignment and cocoercive loading. It would be valuable to optimize the graph topology to allow for larger stepsizes. 

Finally, we only established the ergodic rate of $O(1/k)$ for the pure resolvent splitting. It is interesting to propose some gap or objective function for the general problem \eqref{primal} and the $(u,w,s)$-prototype \eqref{proto}, which also has the ergodic rate of $O(1/k)$.

\section{Acknowledgements}
We are gratefully indebted to Prof. Yuchao Tang (Guangzhou University) for sharing his preprint \cite{tyc_pds} with us, and to Prof. Xiaokai Chang (Lanzhou University of Technology) for his valuable comments.
We did not use AI to develop the ideas or proofs in this paper. AI tools were used to check grammar and typos.

\bibliographystyle{siam}

\small{
\bibliography{refs}
}

\appendix 
\section{Proof of Lemma \ref{l_residual}}
\begin{proof}
We split the residual $\Delta$ as 
 $\Delta =  \Delta_L  +\Delta_C +\Delta_D$,
where $ \Delta_L$, $\Delta_C$, and $\Delta_D$ denote the residuals caused by $L_j^* s_j^k$, $\beta_l$-cocoercive $C_l$ and $\nu_j$-cocoercive $D_j^{-1}$, respectively.  Let us compute each term individually. 

First, $ \Delta_L$ is expressed in terms of levels as 
$ \Delta_L = \sum_{q\ge 0} \Delta_{L,q}$, where
\begin{align*}
\Delta_{L,q} &=  \sum_{i_1=1}^{ |\cI_1|} \cdots 
\sum_{i_q=1}^{ |\cI_{q,i_1,\ldots,i_{q-1}}|}
  \Big\langle \sum_{j \in\cJ_{i_1,\ldots,i_q}} 
  L_{j}^* \big(  (s_{1,j}  -  \stilde_{1,j})  - (s_{2,j}  -  \stilde_{2,j}) 
  \big) \\
  &  +  \sum_{j \in\cJ_{c,i_1,\ldots,i_q}} L_{j}^* 
  \big( (\stilde_{1,j}-  s_{1,j} ) - (\stilde_{2,j}-  s_{2,j}) \big)  
   \Big|  \utilde_{1, i_{q}} - \utilde_{2, i_{q}} \Big\rangle,  
\end{align*}
where the notational conventions follow Remarks \ref{rmk_notation} and \ref{rmk_notation2}, including the case of $q=0$. To compute the summation, we first consider $ \Delta_{L,0}+\Delta_{L,1} $.  Since $\mcup_{i_1=1}^{|\cI_1|}  \cJ_{c,i_1} = \cJ_{0} $,  we have
$\sum_{j \in\cJ_0}  = 
\sum_{i_1=1}^{|\cI_1|} \sum_{j \in\cJ_{c,i_1}} $, and thus,
\[
\begin{aligned}
& \Delta_{L,0} + \Delta_{L,1}  \\
 = & \Big\langle \sum_{i_1=1}^{|\cI_1|} \sum_{j \in\cJ_{c,i_1}} L_{j}^* \big(  (s_{1,j}  -  \stilde_{1,j})  - (s_{2,j}  -  \stilde_{2,j})   \big)  \Big|  \utilde_{1,0} - \utilde_{2,0} \Big\rangle  \\
+ & \sum_{i_1=1}^{ |\cI_1|}
  \Big\langle \sum_{j \in\cJ_{c,i_1}} L_{j}^* 
   \big(  (  \stilde_{1,j} - s_{1,j} )  - (\stilde_{2,j} - s_{2,j})   \big) 
  \Big|   \utilde_{1,i_1} - \utilde_{2,i_1}     \Big\rangle \\
  +& \sum_{i_1=1}^{ |\cI_1|}
  \Big\langle  \sum_{j \in\cJ_{i_1}} L_{j}^* 
  \big(  (s_{1,j}  -  \stilde_{1,j})  - (s_{2,j}  -  \stilde_{2,j})   \big)    
   \Big| \utilde_{1,i_1} - \utilde_{2,i_1} \Big\rangle \\
   = & \sum_{i_1=1}^{|\cI_1|} \Big\langle  \sum_{j \in\cJ_{c,i_1}} L_{j}^* \big(  (s_{1,j}  -  \stilde_{1,j})  - (s_{2,j}  -  \stilde_{2,j})   \big) \Big| 
   ( \utilde_{1,0} - \utilde_{1,i_1})
   - ( \utilde_{2,0} - \utilde_{2,i_1}) \Big\rangle  \\
  + & \underbrace{  \sum_{i_1=1}^{ |\cI_1|}
  \Big\langle  \sum_{j \in\cJ_{i_1}} L_{j}^* 
  \big(  (s_{1,j}  -  \stilde_{1,j})  - (s_{2,j}  -  \stilde_{2,j})   \big)    
   \Big| \utilde_{1,i_1} - \utilde_{2,i_1} \Big\rangle }
   _{: = \Delta_{L, i_1}  } , 
   \end{aligned}
\]
where the term $\Delta_{L, i_1} $ can be combined with some term in $\Delta_{L,2} $, using $\mcup_{i_2=1}^{|\cI_{2,i_1}|} \cJ_{c,i_1, i_{2}}
= \cJ_{i_1} $ and $\sum_{j \in\cJ_{i_1}}  = \sum_{i_2=1}^{|\cI_{2,i_1}|} \sum_{j \in\cJ_{c,i_1,i_2}}  $. 
The same derivation goes for the rest $q\ge 2$,  based on  the general identity
$\mcup_{i_q=1}^{|\cI_{q,i_1,\ldots,i_{q-1}}|} \cJ_{c,i_1,\ldots, i_{q}} = \cJ_{i_1,\ldots, i_{q-1}} $. Finally the sum $\Delta_L^\star$, by induction, becomes:
{\small \[
\begin{aligned}
\Delta_{L}  & =  \sum_{q \ge 1} \bigg(  
\sum_{i_1=1}^{|\cI_1|} \cdots 
\sum_{i_{q}=1}^{|\cI_{q,i_1,\ldots,i_{q-1}}|} 
\Big\langle  \sum_{j \in\cJ_{c,i_1,\ldots,i_{q}}} L_{j}^* \big(  (s_{1,j}  -  \stilde_{1,j})  - (s_{2,j}  -  \stilde_{2,j})   \big) 
\Big|  ( \utilde_{1,i_{q-1}} - \utilde_{1, i_q})
   - ( \utilde_{2, i_{q-1}} - \utilde_{2, i_q})
\Big\rangle \bigg) \\
& =  \sum_{q \ge 1} \bigg(  
\sum_{i_1=1}^{|\cI_1|} \cdots 
\sum_{i_{q}=1}^{|\cI_{q,i_1,\ldots,i_{q-1}}|} 
\Big\langle  \sum_{j \in\cJ_{c,i_1,\ldots,i_{q}}} L_{j}^* 
\big(  (s_{1,j}  -  \stilde_{1,j})  - (s_{2,j}  -  \stilde_{2,j})   \big)  \Big| (z_{1, i_q}  - \ztilde_{1, i_q} )
- (z_{2, i_q}  - \ztilde_{2, i_q} )  \Big\rangle \bigg)
\quad \textrm{by $z_i$-step of \eqref{proto_T_relax}}  \\
& = \sum_{j=1}^{m} \Big\langle  L_{j}^*
\big(  (s_{1,j}  -  \stilde_{1,j})  - (s_{2,j}  -  \stilde_{2,j})   \big) \Big| (z_{1, i}  - \ztilde_{1, i} )
- (z_{2, i}  - \ztilde_{2, i} )  \Big\rangle 
\quad \textrm{where the associated $i\in [n-1]$ is such that $j \in\cJ_{c,i}$}  \\
& \ge - \sum_{j=1}^m \bigg( \frac{1}{4 \tau_{j} }
 \Big\|  L_{j}^* \big(  (s_{1,j}  -  \stilde_{1,j})  - (s_{2,j}  -  \stilde_{2,j})   \big)  \Big\|^2 \bigg)  
- \sum_{\substack{ i\in [n-1] \\ \cJ_{c,i} \ne\varnothing}} 
\bigg( \Big( \underbrace{   \sum_{j\in \cJ_{c,i}} \tau_j }
_{:=\tau_i} \Big)  \Big\|   (z_{1, i}  - \ztilde_{1, i} )
- (z_{2, i}  - \ztilde_{2, i} ) \Big\|^2  \bigg) \\
& \ge - \sum_{j=1}^m \bigg( \frac{\|  L_{j} \|^2 } {4 \tau_{j} }
 \big\|  (s_{1,j}  -  \stilde_{1,j})  - (s_{2,j}  -  \stilde_{2,j})   \big\|^2 \bigg)  
- \sum_{i=1}^{n-1} \Big( \tau_i \big\|   (z_{1, i}  - \ztilde_{1, i} ) - (z_{2, i}  - \ztilde_{2, i} ) \big\|^2  \Big) , \\
\end{aligned}
\]}%
where the last equality follows from rearranging the summation according to the index  $j$. Note that  $\tau_i=0$, whenever $\cJ_{c,i} = \varnothing$.

Second, we consider $\Delta_C$. 
Denoting  $C_{i_{q}} := C_{i_1,\ldots,i_{q}} := \sum_{l \in\cP_{i_1,\ldots,i_{q}}} C_l $ for brevity, then 
\[
\begin{aligned}
\Delta_{C}  & =    \sum_{q\ge 1} \bigg(  \sum_{i_1=1}^{ |\cI_1|} \ldots 
   \sum_{i_{q}=1}^{ |\cI_{q,i_1,\ldots,i_{q-1}}|}
  \Big\langle C_{i_{q}} \utilde_{1,i_{q-1}} - 
  C_{i_{q}} \utilde_{2,i_{q-1}} 
   \Big|  \utilde_{1,i_{q}} - \utilde_{2,i_{q}} \Big\rangle \bigg) \\
   & \ge    -  \sum_{q\ge 1}  \bigg(  \sum_{i_1=1}^{ |\cI_1|} \ldots    \sum_{i_{q}=1}^{ |\cI_{q,i_1,\ldots,i_{q-1}}|}
   \frac{1}{4\beta_{i_{q}}}    
  \big\|(\utilde_{1,i_{q-1}} - \utilde_{1,i_{q}})
  -(\utilde_{2,i_{q-1}} - \utilde_{2,i_{q}})  \big\|^2 \bigg)  \quad \textrm{by Fact \ref{f_coco}-(ii)}\\
& =   -\sum_{\substack{ i\in [n-1] \\ \cP_{i} \ne\varnothing}} \frac{1}{4 \beta_{i} }  \Big\| 
  (z_{1, i}  - \ztilde_{1, i} ) - (z_{2, i}  - \ztilde_{2, i} ) \Big\|^2  \\
& =   -\sum_{i=1}^{n-1} \frac{1}{4 \beta_{i} }  \Big\| 
  (z_{1, i}  - \ztilde_{1, i} ) - (z_{2, i}  - \ztilde_{2, i} ) \Big\|^2 ,  \\
\end{aligned}
\]
where $\beta_i$ denotes the cocoercivity of $C_i = \sum_{l\in \cP_i} C_l$ that are loaded on the non-root node $u_i$, and is given as $\beta_i = \frac{1}{\sum_{l\in\cP_i} \beta_l^{-1}}$ by Fact \ref{f_coco}-(i). We set $\frac{1}{4 \beta_{i} } = 0$, whenever $\cP_i =\varnothing$ in the last equality.

Third, consider $\Delta_D$. Since $D_j^{-1}$ is $\nu_j$-cocoercive, $\Delta_D$, by Fact \ref{f_coco}-(ii), is given as 
\[
\begin{aligned}
\Delta_{D}  &=   \sum_{j=1}^m
  \Big\langle  D_j^{-1} s_{1,j} - D_j^{-1} s_{2,j}  
   \Big|  \stilde_{1,j} - \stilde_{2,j} \Big\rangle 
  \ge - \sum_{j=1}^m \frac{1}{4\nu_j} 
   \big\|  (s_{1,j}  -  \stilde_{1,j})  - (s_{2,j}  -  \stilde_{2,j})    \big\|^2 . \\
\end{aligned}
\]
In particular, we set $\frac{1}{4\nu_j} =0$ if $D_j = \cN_{\{0\}} $ such that $B_j \square D_j = B_j$, because this $D_j^{-1}$ by definition corresponds to $\nu_j = \infty$.
Finally, the total residual $\Delta$ is summarized as Lemma \ref{l_residual}.
\end{proof}

\section{Proof of Proposition \ref{p_ine}}
\begin{proof}
To cope with the residual term, the proof is slightly modified and adapted from  \cite[Sect. 2]{fxue_coam},  \cite[Lemma 1]{fxue_optl} and \cite[Theorem 6.1]{hbs_yxm_2015}. Using Lemma \ref{l_ppm}, we develop for any $x_1, x_2 \in \cH^{2n-1} \times \prod_{j=1}^m \cK_j$ that 
\begin{align*}
0 & \le  \big\langle \cA \xtilde_1 -\cA \xtilde_2  \big| 
\xtilde_1 -  \xtilde_2  \big\rangle   \\
&=  \big\langle (x_1 - \xtilde_1) -  (x_2 - \xtilde_2) \big| \xtilde_1 - \xtilde_2  \big\rangle_ \cQ
 - \underbrace{ \big\langle \cR x_1  -  \cR x_2  \big| \xtilde_1 - \xtilde_2 \big\rangle}_\textrm{ $= \Delta$ defined in \eqref{delta}} .
\end{align*}
Adding $\big\|  (x_1 - \xtilde_1) -  (x_2 - \xtilde_2) \big\|_\cQ^2$ on both sides yields
\[
\Delta + \big\|  (x_1 - \xtilde_1) -  (x_2 - \xtilde_2) \big\|_\cQ^2
\le  \big\langle  x_2 - x_1    \big| 
  (x_2 - \xtilde_2) -  (x_1 - \xtilde_1)\big\rangle_\cQ.
\]
By the structure of $\cQ$ and the relaxation of \eqref{proto_T_relax}, the above inequality becomes:
\begin{align} \label{aa}
& \Delta + \sum_{i=1}^{n-1} \big\| ( z_{1,i} - \ztilde_{1,i})  -  (z_{2,i} -\ztilde_{2,i} )  \big\|_{M_i}^2
+ \sum_{j=1}^{m} \big\| ( s_{1,j} - \stilde_{1,j})  -  (s_{2,j} -\stilde_{2,j} )  \big\|_{N_j}^2 \nonumber \\
 = & \Delta + \sum_{i=1}^{n-1} \frac{1}{\theta_i^2}
 \big\| ( z_{1,i} -  z_{1,i}^+)  -  (z_{2,i}  - z_{2,i}^+ )  \big\|_{M_i}^2
+ \sum_{j=1}^{m} \frac{1}{\zeta_j^2} \big\| ( s_{1,j} - s_{1,j}^+)  -  (s_{2,j} - s_{2,j}^+ )  \big\|_{N_j}^2 \nonumber \\
\le &  \sum_{i=1}^{n-1}  \big\langle  z_{2,i} - z_{1,i}    \big| 
  (z_{2,i} -\ztilde_{2,i})  -  ( z_{1,i}-\ztilde_{1,i}) \big\rangle_{M_i}
  + \sum_{j=1}^{m}  \big\langle   s_{2,j} - s_{1,j}     \big| 
  (s_{2,j} -\stilde_{2,j})  -  ( s_{1,j}- \stilde_{1,j}) \big\rangle_{N_j} \nonumber  \\
= &  \sum_{i=1}^{n-1} \frac{1}{\theta_i} 
\big\langle z_{2,i} - z_{1,i}  \big| 
  (z_{2,i} -z_{2,i}^+)  -  ( z_{1,i}-z_{1,i}^+ )\big\rangle_{M_i}
  + \sum_{j=1}^{m}  \frac{1}{\zeta_j} 
    \big\langle   s_{2,j} - s_{1,j}   \big| 
  (s_{2,j} - s_{2,j}^+)  -  ( s_{1,j}- s_{1,j}^+) \big\rangle_{N_j}.
\end{align}

On the other hand, based on the identity $\|a\|^2-\|b\|^2 = 2\langle a|a-b\rangle - \|a-b\|^2$, we consider
\begin{align*}
& \sum_{i=1}^{n-1} \frac{1}{\theta_i}  \Big( \big\|  z_{2,i} - z_{1,i} \big\|_{M_i}^2 - \big\|  z_{2,i}^+ - z_{1,i}^+\big\|_{M_i}^2 \Big)
+ \sum_{j=1}^{m} \frac{1}{\zeta_j} \Big( \big\|  s_{2,j} - s_{1,j} \big\|_{N_j}^2
- \big\|  s_{2,j}^+ - s_{1,j}^+  \big\|_{N_j}^2 \Big) \\
= &  \sum_{i=1}^{n-1} \frac{2}{\theta_i} \big\langle 
 z_{2,i} - z_{1,i}   \big|
( z_{2,i} - z_{2,i}^+  ) - ( z_{1,i} - z_{1,i}^+ ) \big\rangle_{M_i}
-  \sum_{i=1}^{n-1} \frac{1}{\theta_i} \big\|
( z_{2,i} - z_{2,i}^+  ) - ( z_{1,i} - z_{1,i}^+ ) \big\|_{M_i}^2 \\
+ & \sum_{j=1}^{m} \frac{2}{\zeta_j} \big\langle s_{2,j} - s_{1,j} \big|
( s_{2,j} - s_{2,j}^+  ) - ( s_{1,j} - s_{1,j}^+ )  \big\rangle_{N_j}
- \sum_{j=1}^{m} \frac{1}{\zeta_j} \big\|
( s_{2,j} - s_{2,j}^+  ) - ( s_{1,j} - s_{1,j}^+ )  \big\|_{N_j}^2 \\
\ge & 2\Delta+  \sum_{i=1}^{n-1} \Big( \frac{2}{\theta_i^2}
-\frac{1}{\theta_i} \Big) \big\| ( z_{2,i} - z_{2,i}^+  ) - ( z_{1,i} - z_{1,i}^+ )\big\|_{M_i}^2 
+  \sum_{j=1}^{m} \Big( \frac{2}{\zeta_j^2}
-\frac{1}{\zeta_j} \Big)  \big\|( s_{2,j} - s_{2,j}^+  ) - ( s_{1,j} - s_{1,j}^+ ) \big\|_{N_j}^2 ,
\end{align*}
where the last inequality is due to \eqref{aa}.
Finally, substituting $\Delta$ of Lemma \ref{l_residual} into above completes the proof.
\end{proof}

\section{Proof of Theorem \ref{t_zs}}
\begin{proof}
(i) The proof follows a standard procedure of \cite[Theorem 5.1]{fxue_jota}:

Step-1: For every $(z_i^\star, s_j^\star) \in \Fix \cT_{zs}  $, $\lim_{k\rightarrow \infty} \sum_{i=1}^{n-1} \frac{1}{\theta_i} \big\| 
  z_{i}^{k} - z_{i}^\star \big\|_{M_i }^2 
+ \sum_{j=1}^m   \frac{1}{\zeta_j} \big\|   s_{j}^{k} -  s_{j}^\star  \big\|_{N_j }^2   $ exists;  

Step-2: $\big\{ (z_i^k, s_j^k) \big\}_{k\in\N}$ has at least  one weak sequential cluster point lying in $ \Fix \cT_{zs}$;

Step-3: The weak sequential  cluster point of  $\big\{ (z_i^k, s_j^k) \big\}_{k\in\N}$  is unique.

\vskip.2cm
Step-1: Since $\Fix \cT_{zs,\theta,\zeta}\ne\varnothing$ by Corollary \ref{c_existence}, letting $x_1 = x^k$ and $x_2 = x^\star \in \Fix \cT_{uzs}$ in Proposition \ref{p_ine}, and noting $x_1^+ = x^{k+1}$ and $\xtilde_2 = x_2^+ = x_2$, one obtains:
\be  \label{aqw}
  E_k - E_{k+1}  \ge  \sum_{i=1}^{n-1} \frac{1}{\theta_i^2} \big\|   z_{i}^{k+1} - z_{i}^k \big\|_{
   (2-\theta_i) M_i - ( 2\tau_i  + \frac{1}{2 \beta_{i} } ) I }^2  
+  \sum_{j=1}^{m} \frac{1}{\zeta_j^2} \big\| 
  s_{j}^{k+1} - s_{j}^k \big\|_{
   (2- \zeta_j) N_j - (\frac{\|  L_{j} \|^2 } {2 \tau_{j} } +
     \frac{1}{2 \nu_j} ) I  }^2, 
\ee
where $E_k =  \sum_{i=1}^{n-1} \frac{1}{\theta_i} \big\| 
  z_{i}^{k} - z_{i}^\star \big\|_{M_i }^2 
+ \sum_{j=1}^m   \frac{1}{\zeta_j} \big\|   s_{j}^{k} -  s_{j}^\star  \big\|_{N_j }^2 $. This shows that the sequence $\{E_k\}_{k\in\N} $ is  non-increasing under the conditions of Corollary \ref{c_demi}, and bounded from below (non-negative), and thus its limit exists.

Step-2: The sequence $\big\{ (z_i^k, s_j^k)\big\}_{k\in\N}$ is bounded by \eqref{aqw}, and then, it must possess at least one weak sequential cluster point $(z_i^*, s_j^*)$. It implies that  there exists a subsequence  $\big\{ (z_i^{k_l}, s_j^{k_l}) \big\}_{k\in\N}$\footnote{The symbol $l$ always stands for the index of cocoercive $C_l$ in \eqref{p} throughout this paper. We borrow $l$ for the moment (in this proof) to denote the subsequence index.} that weakly converges to $(z_i^*, s_j^*)$, as $k_l\rightarrow \infty$. 
We here need to show that $ (z_i^*, s_j^*) \in \Fix \cT_{zs}$, and more generally,  every weak sequential cluster point of  
$\{ (z_i^k, s_j^k) \}_{k\in\N}$ belongs to $ \Fix \cT_{zs} $. To this end, summing up \eqref{aqw} from $k=0$ to $K-1$, and taking $K \rightarrow \infty$, we have
\begin{align}  \label{bb}
&  \sum_{k=0}^{\infty} \bigg( 
\sum_{i=1}^{n-1} \frac{1}{\theta_i^2} \big\| 
  z_{i}^{k+1} - z_{i}^k \big\|_{
   (2-\theta_i) M_i - ( 2\tau_i  + \frac{1}{2 \beta_{i} } ) I }^2  
+  \sum_{j=1}^{m} \frac{1}{\zeta_j^2} \big\| 
  s_{j}^{k+1} - s_{j}^k \big\|_{
   (2- \zeta_j) N_j - (\frac{\|  L_{j} \|^2 } {2 \tau_{j} } +
     \frac{1}{2 \nu_j} ) I  }^2 \bigg)   \nonumber \\
& \le   E_0   < +\infty.
\end{align}
Under the conditions in Corollary \ref{c_demi}, \eqref{bb}  implies that  $  (z_i^k - z_i^{k+1}, s_j^k - s_j^{k+1})  = (\cI -   \cT_{zs,\theta,\zeta})  (z_i^k, s_j^k ) \rightarrow  0$, and also $ (\cI -   \cT_{zs,\theta,\zeta}) (z_i^{k_l}, s_j^{k_l} )  \rightarrow  0$ as $k_l \rightarrow \infty$. 
Combining with $(z_i^{k_l}, s_j^{k_l})  \weak (z_i^*, s_j^*) $ and the 
$(\frac{1}{\theta_i} M_i, \frac{1}{\zeta_j}  N_j)$-demiclosedness of $\cI -  \cT_{zs,\theta,\zeta}$ (shown in Corollary \ref{c_demi}), we obtain $ (\cI - \cT_{zs,\theta,\zeta})  (z_i^*, s_j^*) = 0$, i.e., $ (z_i^*, s_j^*) \in \Fix  \cT_{zs,\theta,\zeta}$. 
Since $\big\{ (z_i^{k_l}, s_j^{k_l}) \big\}_{k_l\in\N}$ is an arbitrary weakly convergent subsequence, we conclude that every weak sequential cluster point of $\{ (z_i^{k }, s_j^{k }) \}_{k \in\N}$ lies in $\Fix \cT_{zs,\theta,\zeta} $. 

\vskip.2cm
Step-3: We need to show that $\big\{ (z_i^{k }, s_j^{k })  \big\}_{k \in\N}$ cannot have two distinct weak sequential cluster point in $ \Fix \cT_{zs,\theta,\zeta}$. Indeed, let $ (z_{1,i}^*, s_{1,j}^*),  (z_{2,i}^*, s_{2,j}^*) \in \Fix \cT_{zs,\theta,\zeta} $ be two cluster points of  $\big\{ (z_i^{k }, s_j^{k }) \big\}_{k \in\N}$. Since  $\lim_{k\rightarrow \infty} E_k $  exists as proved in Step-1, set $\xi_1 = \lim_{k\rightarrow \infty} \sum_{i=1}^{n-1} \frac{1}{\theta_i} \big\| 
  z_{i}^{k} - z_{1,i}^*  \big\|_{M_i }^2 
+ \sum_{j=1}^m   \frac{1}{\zeta_j} \big\|   s_{j}^{k} -  s_{1,j}^*  \big\|_{N_j }^2 $, and  $\xi_2 = \lim_{k\rightarrow \infty} \sum_{i=1}^{n-1} \frac{1}{\theta_i} \big\| 
  z_{i}^{k} - z_{2,i}^*  \big\|_{M_i }^2 
+ \sum_{j=1}^m   \frac{1}{\zeta_j} \big\|   s_{j}^{k} -  s_{2,j}^*  \big\|_{N_j }^2 $. 
Take a subsequence $\big\{ (z_i^{k_l}, s_j^{k_l}) \big\}_{k \in\N}$ weakly converging to  $ (z_{1,i}^*, s_{1,j}^*) $, as $k_l \rightarrow \infty$. From the identity $\|a-b\|^2 - \|a-c\|^2 = \|b-c\|^2 +2\langle b-c| c-a\rangle$, we have:
\begin{align*}
&\sum_{i=1}^{n-1} \frac{1}{\theta_i} \Big( \big\| 
  z_{i}^{k_l} - z_{1,i}^*  \big\|_{M_i }^2 
  - \big\|   z_{i}^{k_l} - z_{2,i}^*  \big\|_{M_i }^2 \Big)
+ \sum_{j=1}^m   \frac{1}{\zeta_j} \Big( 
 \big\|   s_{j}^{k_l} -  s_{1,j}^*  \big\|_{N_j }^2  -
  \big\|   s_{j}^{k_l} -  s_{2,j}^*  \big\|_{N_j }^2 \Big)\\
= & \sum_{i=1}^{n-1} \frac{1}{\theta_i}  \big\| 
   z_{1,i}^* - z_{2,i}^*  \big\|_{M_i }^2 
+ \sum_{j=1}^m   \frac{1}{\zeta_j}  
 \big\|   s_{1,j}^*  -  s_{2,j}^*  \big\|_{N_j }^2   \\
+ & 2 \sum_{i=1}^{n-1} \frac{1}{\theta_i}  \big\langle   
   z_{1,i}^* - z_{2,i}^*   \big| 
   z_{2,i}^* - z_{i}^{k_l} \big\rangle_{M_i}
+ 2 \sum_{j=1}^{m} \frac{1}{\zeta_j}  \big\langle   
   s_{1,j}^* - s_{2,j}^*   \big| 
   s_{2,j}^* - s_{j}^{k_l} \big\rangle_{N_j}.
\end{align*}
Taking $k_l \rightarrow \infty$ on both sides, the last two terms become $ -2 \big(  \sum_{i=1}^{n-1} \frac{1}{\theta_i}  \big\|    z_{1,i}^* - z_{2,i}^*  \big\|_{M_i }^2 
+ \sum_{j=1}^m   \frac{1}{\zeta_j}  
 \big\|   s_{1,j}^*  -  s_{2,j}^*  \big\|_{N_j }^2 \big)$, then we deduce that  $\xi_1  - \xi_2  = - \sum_{i=1}^{n-1} \frac{1}{\theta_i}  \big\|    z_{1,i}^* - z_{2,i}^*  \big\|_{M_i }^2 
- \sum_{j=1}^m   \frac{1}{\zeta_j}  \allowbreak
 \big\|   s_{1,j}^*  -  s_{2,j}^*  \big\|_{N_j }^2$. Similarly, take another subsequence $\{  x^{k_{l'}} \}_{k \in\N}$ weakly converging to   $ (z_{2,i}^*, s_{2,j}^*)$, which finally yields that $\xi_1  - \xi_2  = \sum_{i=1}^{n-1} \frac{1}{\theta_i}  \big\|    z_{1,i}^* - z_{2,i}^*  \big\|_{M_i }^2 
+ \sum_{j=1}^m   \frac{1}{\zeta_j}  
 \big\|   s_{1,j}^*  -  s_{2,j}^*  \big\|_{N_j }^2$.  Consequently, $\sum_{i=1}^{n-1} \frac{1}{\theta_i}  \big\|    z_{1,i}^* - z_{2,i}^*  \big\|_{M_i }^2 
+ \sum_{j=1}^m   \frac{1}{\zeta_j}  
 \big\|   s_{1,j}^*  -  s_{2,j}^*  \big\|_{N_j }^2 = 0$, then the conditions in Corollary \ref{c_demi} imply that $ z_{1,i}^* = z_{2,i}^*$ and $ s_{1,j}^*  =  s_{2,j}^* $. This shows the uniqueness of the weak sequential cluster point.

Finally, combining the above 3 steps and \cite[Lemma 2.46]{plc_book}, we conclude the weak convergence of $\big\{ (z_i^{k }, s_j^{k }) \big\}_{k \in\N}$, and denote the weak limit as $(z_i^\star, s_j^\star ) \in \Fix \cT_{zs,\theta,\zeta} = \Fix \cT_{zs}$ by Fact \ref{f_T}. 

(ii)  Since $(z_i^\star, s_j^\star ) \in \Fix \cT_{zs}$, the proof is completed by  Proposition \ref{p_encoding}-(i).
\end{proof}

\section{Proof of Proposition \ref{p_rate}}
\begin{proof}
Letting $x_1 = x^k$ and $x_2 = x^{k-1}$ in Proposition \ref{p_ine} and noting that  $x_1^+ = x^{k+1}$ and $x_2^+ = x^{k}$, we obtain:
\begin{align} \label{zy}
& \sum_{i=1}^{n-1} \frac{1}{\theta_i}  \big( \big\| z_i^{k-1} -z_i^{k} \big\|_{M_i}^2 - \big\| z_i^{k} -z_i^{k+1} \big\|_{M_i}^2 \big)
+ \sum_{j=1}^{m} \frac{1}{\zeta_j} \big( \big\| s_j^{k-1} - s_j^{k} \big\|_{N_j}^2
- \big\| s_j^{k} - s_j^{k+1} \big\|_{N_j}^2 \big) \nonumber \\
\ge &   \sum_{i=1}^{n-1}  \frac{1}{\theta_i^2}
  \big\| z_i^{k-1} - z_i^{k} - ( z_i^{k} - z_i^{k+1} ) \big\|_{ (2-\theta_i) M_i - ( 2\tau_i  + \frac{1}{2 \beta_{i} } ) I }^2 
  \nonumber \\
+&  \sum_{j=1}^{m}  \frac{1}{\zeta_j^2}
 \big\| s_j^{k-1} - s_j^{k} - ( s_j^{k} - s_j^{k+1} ) \big\|_{(2- \zeta_j) N_j - (\frac{\|  L_{j} \|^2 } {2 \tau_{j} } +
     \frac{1}{2 \nu_j} ) I}^2.
\end{align}     

On the other hand, summing up \eqref{aqw} from $k=0$ to $K-1$ yields
\begin{align*} 
& \sum_{i=1}^{n-1} \frac{1}{\theta_i} \big\| 
  z_{i}^{0} - z_{i}^\star \big\|_{M_i }^2 
+ \sum_{j=1}^m   \frac{1}{\zeta_j} \big\|   s_{j}^{0} -  s_{j}^\star  \big\|_{N_j }^2 \\
 \ge  &  \sum_{k=0}^{K-1} \bigg(
\sum_{i=1}^{n-1} \frac{1}{\theta_i^2} \big\| 
  z_{i}^{k+1} - z_{i}^k \big\|_{
   (2-\theta_i) M_i - ( 2\tau_i  + \frac{1}{2 \beta_{i} } ) I }^2  
+  \sum_{j=1}^{m} \frac{1}{\zeta_j^2} \big\| 
  s_{j}^{k+1} - s_{j}^k \big\|_{
   (2- \zeta_j) N_j - (\frac{\|  L_{j} \|^2 } {2 \tau_{j} } +
     \frac{1}{2 \nu_j} ) I  }^2   \bigg) \\ 
\ge   & (\xi-1) \cdot \sum_{k=0}^{K-1} \bigg(
\sum_{i=1}^{n-1} \frac{1}{\theta_i }   \big\| 
  z_{i}^{k+1} - z_{i}^k \big\|_{ M_i  }^2  
+  \sum_{j=1}^{m} \frac{1}{\zeta_j }  \big\| 
  s_{j}^{k+1} - s_{j}^k \big\|_{  N_j  }^2   \bigg) \\
\ge   & (\xi-1) \cdot K
 \bigg( \sum_{i=1}^{n-1} \frac{1}{\theta_i }   \big\| 
  z_{i}^{K} - z_{i}^{K-1} \big\|_{ M_i  }^2  
+  \sum_{j=1}^{m} \frac{1}{\zeta_j }  \big\| 
  s_{j}^{K} - s_{j}^{K-1} \big\|_{  N_j  }^2   \bigg) ,
\end{align*}
where the last inequality is due to the fact that the sequence $
\big\{ \sum_{i=1}^{n-1} \frac{1}{\theta_i }   \big\| 
  z_{i}^{k+1} - z_{i}^k \big\|_{ M_i  }^2  
+  \sum_{j=1}^{m} \frac{1}{\zeta_j }  \big\| 
  s_{j}^{k+1} - s_{j}^k \big\|_{  N_j  }^2 \big\}_{k\in\N} $ is non-increasing by \eqref{zy}, which also proves the rate of $o(1/k)$. 
\end{proof}

\section{Proof of Lemma \ref{l_rho}}
\begin{proof}
Denote $ (u_{1,0}^+, u_{1,i}^+, z_{1,i}^+ ) = \cT_{uz,\theta} (u_{1,0}, u_{1,i}, z_{1,i} )$ and   
$(u_{2,0}^+, u_{2,i}^+, z_{2,i}^+ )  = \cT_{uz,\theta} (u_{2,0},
\allowbreak  u_{2,i}, z_{2,i} )$, where $ \cT_{uz,\theta}$ is given as \eqref{pure_drs}.
Based on the boundedness and closedness of each $M_i$, and the $z_{i_q}$-step of \eqref{pure_drs}, Lemma \ref{l_rho} is essentially equivalent to the existence of $\rho_s$, such that $\|u_{1,s}^+ - u_{2,s}^+\|^2 \le \rho_s \cdot \sum_{i=1}^{n-1} 
\big\| z_{1,i} - z_{2,i} \big\|_{M_i}^2 $, $\forall s\in \{0\} \mcup [n-1]$. Then, we aim to show the equivalent statement that $\exists \rho_{i_q} >0$, such that $\|u_{1,i_q}^+ - u_{2, i_q}^+\|^2 \le \rho_{i_q} \cdot \sum_{i=1}^{n-1} 
\big\| z_{1,i} - z_{2,i} \big\|_{M_i}^2 $, $\forall q\ge 0$. 

Based on Fact \ref{f_res}, we obtain 
\[
\big\| u_{ 1,0 }^+ - u_{ 2,0 }^+ \big\|_{S_0}^2 \le 
\Big\| \sum_{i_{1} =1}^{ |\cI_{1} | } M_{i_1 } \big(
 z_{1, i_1 }   -  z_{2, i_1 } \big) \Big\|_{S_{0}^{-1}}^2
  \le  \rho_0 \cdot \sum_{i_{1} =1}^{ |\cI_{1} | }
\Big\|  z_{1, i_1 } -  z_{2, i_1 } \Big\|_{S_{0}^{-1}}^2 , 
\]
and
\begin{align*}
& \big\| u_{ 1, i_q }^+ - u_{ 2, i_q }^+ \big\|_{S_{i_q}}^2  \\
\le & \Big\| M_{ i_{q} } \big( 2 u_{1, i_{q-1} }^+  -2 u_{2, i_{q-1} }^+ -z_{1, i_{q} }   + z_{2, i_{q} } \big)  +  \sum_{i_{q+1} =1}^{ |\cI_{q+1,i_1,\ldots, i_q} | }
M_{ i_{q+1} }  \big( z_{1, i_{q+1} }   -  z_{2, i_{q+1} } \big) 
 \Big\|_{S_{i_q}^{-1}}^2 \\
\le &  \rho_{i_q}  \cdot  \Big( 
 \big\|    u_{1, i_{q-1} }^+  -  u_{2, i_{q-1} }^+ \big\|_{S_{i_q}^{-1}}^2  + \big\| z_{1, i_{q} }   - z_{2, i_{q} } \big\|_{S_{i_q}^{-1}}^2  +   \sum_{i_{q+1} =1}^{ |\cI_{q+1,i_1,\ldots, i_q} | } 
 \big\|   z_{1, i_{q+1} }   -  z_{2, i_{q+1} }  \big\|_{S_{i_q}^{-1}}^2 \Big), 
\end{align*}
where $\rho_0 =  \sum_{i_{1} =1}^{ |\cI_{1} | }  \| M_{i_1 } \| _{S_{0}^{-1}}^2$ and $\rho_{i_q} = 5\big\| M_{ i_{q} }  \big\|_{S_{i_q}^{-1}}^2
+ \sum_{i_{q+1} =1}^{ |\cI_{q+1,i_1,\ldots, i_q} | }
\big\| M_{ i_{q+1} }\big\|_{S_{i_q}^{-1}}^2$. 
Since the tree has only finite number of nodes, and each $M_i$ is uniformly positive definite and bounded, all the metric norms, such as $\|\cdot\|_{S_i}$, $\|\cdot\|_{S_i^{-1}}$ and $\|\cdot\|_{M_i}$, are equivalent to the ordinary norm $\|\cdot\|$. Then  the first term  $\big\|    u_{1, i_{q-1} }^+  -  u_{2, i_{q-1} }^+ \big\|_{S_{i_q}^{-1}}^2 $ can be recursively traced back (from leaves to root) to and upper bounded by  $\rho' \cdot\big\| u_{ 1,0 }^+ - u_{ 2,0 }^+ \big\|_{S_0}^2$ for some $\rho'>0$, according to the iteration and the ancestral line of tree structure. In all, one can always find some $\rho$, such that Lemma \ref{l_rho} holds.
\end{proof}

\section{Proof of Lemma \ref{l_pi}}
\begin{proof}
(i) Based on the convexities of $\tilde{f}_0 = f_0 -\langle \cdot|a \rangle$ and $f_i$ for $i\in [n-1]$ and Lemma \ref{l_ppm_w}, we have
\begin{align} \label{f_0}
\tilde{f}_0(u_0) & \ge \tilde{f}_0(u_0^{k+1}) +\big\langle \partial \tilde{f}_0(u_0^{k+1}) \big| u_0 -  u_0^{k+1} \big\rangle \nonumber \\
&= \tilde{f}_0(u_0^{k+1}) -\Big\langle \sum_{i_1 \in \cI_1} M_{i_1} w_{i_1}^{k+1}   \Big| u_0 -  u_0^{k+1} \Big\rangle
 - \big\langle \cQ_{u_0}  (x^{k+1}- x^{k})   \big| u_0 -  u_0^{k+1} \big\rangle ,
\end{align}
and $\forall q\ge 1$,
\begin{align} \label{f_i}
f_{i_q}(u_{i_q}) & \ge f_{i_q} (u_{i_q}^{k+1}) +\big\langle \partial f_{i_q} (u_{i_q}^{k+1}) \big| u_{i_q}  -  u_{i_q}^{k+1} \big\rangle \nonumber \\
&= f_{i_q}(u_{i_q}^{k+1}) -\Big\langle \sum_{i_{q+1} \in \cI_{q+1,i_1,\ldots,i_q}} M_{i_{q+1}} w_{i_{q+1}}^{k+1}   \Big| u_{i_q}  -  u_{i_q}^{k+1} \Big\rangle + \Big\langle  M_{i_{q}} w_{i_{q}}^{k+1}   \Big| u_{i_q}  -  u_{i_q}^{k+1} \Big\rangle \nonumber \\
& - \big\langle \cQ_{u_{i_q}}  (x^{k+1}- x^{k})  \big| u_{i_q}  -  u_{i_q}^{k+1} \big\rangle ,
\end{align}
where $\cQ_{u_0}$ and $\cQ_{u_{i_q}}$ stand for the rows of $\cQ$ corresponding to $u_0$ and $u_{i_q}$, respectively.

Substituting \eqref{f_0} and \eqref{f_i} into \eqref{pi} yields 
\begin{align} \label{ws}
\Pi (x^{k+1},x)  = & \cL(u_0^{k+1},u_i^{k+1},w_i  ) - \cL(u_0 ,u_i ,w_i^{k+1} )  \nonumber \\
 = & \tilde{f}_0(u_0^{k+1}) +\sum_{i=1}^{n-1} f_i (u_i^{k+1})
- \tilde{f}_0(u_0) - \sum_{i=1}^{n-1} f_i (u_i)  \nonumber \\
+& \sum_{q\ge 1} \sum_{i_q\in \cI_{q,i_1,\ldots,i_q-1}}
\Big( \big\langle M_{i_q} w_{i_q}  \big|  u_{i_{q-1}}^{k+1} - u_{i_{q}}^{k+1} \big\rangle
- \big\langle M_{i_q}  w_{i_q}^{k+1} \big|  u_{i_{q-1}} - u_{i_{q}} \big\rangle \Big) 
\nonumber \\
\le & E_u  + E_\cQ , 
\end{align}
where
\begin{align*}
E_u  & =   \Big\langle \sum_{i_1 \in \cI_1} M_{i_1} w_{i_1}^{k+1}   \big| u_0 -  u_0^{k+1} \Big\rangle 
+\sum_{q\ge 1}\bigg(  \Big\langle \sum_{i_{q+1} 
\in \cI_{q+1,i_1,\ldots,i_{q}}  } M_{i_{q+1}} w_{i_{q+1}}^{k+1}   \Big| u_{i_q} -  u_{i_q}^{k+1} \Big\rangle \\
& -  \Big\langle M_{i_q} w_{i_{q}}^{k+1}   \Big| u_{i_q}  -  u_{i_q}^{k+1} \Big\rangle \bigg) 
+ \sum_{q\ge 1} \sum_{i_q \in \cI_{q,i_1,\ldots,i_{q-1}} }
\Big( \langle w_{i_q} |  u_{i_{q-1}}^{k+1} - u_{i_{q}}^{k+1} \rangle
- \langle w_{i_q}^{k+1} |  u_{i_{q-1}} - u_{i_{q}} \rangle \Big), 
\\
E_\cQ & = \big\langle \cQ_{u_0}  (x^{k+1}- x^{k})   \big| u_0 -  u_0^{k+1} \big\rangle + \sum_{q\ge 1} \big\langle \cQ_{u_{i_q}}  (x^{k+1}- x^{k})   \big| u_{i_q} -  u_{i_q}^{k+1} \big\rangle \\
&+ \sum_{j=1}^m  \big\langle \cQ_{s_{j}}  (x^{k+1}- x^{k})  \big| s_{j} -  s_{j}^{k+1} \big\rangle  . 
\end{align*}

$E_u$ can be simplified by basic algebraic manipulations as below:
\begin{align*}
E_u & = \sum_{q\ge 1}  \Big\langle  \sum_{i_q \in \cI_{q,i_1,\ldots,i_{q-1}} } M_{i_q} w_{i_q}^{k+1} \Big|  u_{i_{q-1}} - u_{i_{q-1}}^{k+1} \Big\rangle
- \sum_{q\ge 1} \Big\langle M_{i_q}  w_{i_q}^{k+1} \Big|  u_{i_{q}}  - u_{i_{q}}^{k+1} \Big\rangle \\
& +  \sum_{q\ge 1} \sum_{i_q \in \cI_{q,i_1,\ldots,i_{q-1}} }
\Big( \big\langle M_{i_q} w_{i_q}  \big|  u_{i_{q-1}}^{k+1} - u_{i_{q}}^{k+1} \big\rangle
- \big\langle M_{i_q} w_{i_q}^{k+1} \big|  u_{i_{q-1}}  - u_{i_{q}}  \big\rangle \Big) \\
& = \sum_{q\ge 1} \sum_{i_q \in \cI_{q,i_1,\ldots,i_{q-1}} }  \bigg( \big\langle   M_{i_q} w_{i_q}^{k+1} \big|  u_{i_{q-1}} - u_{i_{q-1}}^{k+1} \big\rangle
 +  \big\langle M_{i_q} w_{i_q} \big|  u_{i_{q-1}}^{k+1} - u_{i_{q}}^{k+1} \big\rangle \\
& - \big\langle M_{i_q} w_{i_q}^{k+1} \big|  u_{i_{q-1}} - u_{i_{q}}  \big\rangle \bigg) 
 - \sum_{q\ge 1} \Big\langle M_{i_q}  w_{i_q}^{k+1} \Big|  u_{i_{q}}  - u_{i_{q}}^{k+1} \Big\rangle \\
& = \sum_{q\ge 1} \sum_{i_q \in \cI_{q,i_1,\ldots,i_{q-1}} }   \big\langle   M_{i_q} (w_{i_q}  - w_{i_q}^{k+1} ) \big|   u_{i_{q-1}}^{k+1} \big\rangle
 + \sum_{q\ge 1} \Big(  \big\langle M_{i_q} w_{i_q}^{k+1} \big|  u_{i_{q}}  \big\rangle
- \big\langle M_{i_q} w_{i_q} \big|  u_{i_{q}}^{k+1} \big\rangle \Big) \\
& - \sum_{q\ge 1} \Big\langle M_{i_q}  w_{i_q}^{k+1} \Big|  u_{i_{q}} - u_{i_{q}}^{k+1} \Big\rangle \\
& = \sum_{q\ge 1} \sum_{i_q \in \cI_{q,i_1,\ldots,i_{q-1}} }   \big\langle   M_{i_q} (w_{i_q}  - w_{i_q}^{k+1} ) \big|   u_{i_{q-1}}^{k+1} - u_{i_{q}}^{k+1} \big\rangle \\
& = \sum_{q\ge 1} \sum_{i_q \in \cI_{q,i_1,\ldots,i_{q-1}} }   \big\langle   w_{i_q}  - w_{i_q}^{k+1}  \big| 
M_{i_q} (  u_{i_{q}}^{k} - u_{i_{q}}^{k+1} 
- w_{i_{q}}^{k} + w_{i_{q}}^{k+1} ) \big\rangle 
\quad \textrm{by the $w_{i_q}$-step of \eqref{proto}} \\
& = \sum_{q\ge 1} \sum_{i_q \in \cI_{q,i_1,\ldots,i_{q-1}} }   \big\langle   w_{i_q} - w_{i_q}^{k+1}  \big| 
\cQ_{w_{i_q}} (x^{k+1} - x^{k} ) \big\rangle
\quad \textrm{by the structure of $\cQ$ in Lemma \ref{l_ppm_w}}  \\
& = \sum_{i=1}^{n-1}  \big\langle \cQ_{w_{i}} (x^{k+1} - x^{k} ) \big|   w_{i}  - w_{i}^{k+1}  \big\rangle,
\end{align*}
Thus, we have $E_u+E_\cQ =  \big\langle x^{k+1} - x^{k} \big| x   - x^{k+1}  \big\rangle_\cQ $ by the structure of $\cQ$. Finally, (i) follows from \eqref{ws}.

(ii) From (i) follows by Proposition \ref{p_ine} that 
\[
2\Pi (x^{k+1}, x )  \le    \sum_{i=1}^{n-1}  \Big( \big\| 
  z_{i}^{k} - z_{i}  \big\|_{M_i  }^2 
  -\big\|   z_{i}^{k+1} - z_{i}  \big\|_{M_i  }^2 
-   \big\| z_{i}^{k+1} - z_{i}^k \big\|_{M_i }^2  \Big) ,
\]
where $z_i^k = u_i^k - w_i^k$ and $z_i  = u_i  - w_i $.  Summing up from $k=0$ to $k=K-1$ yields
\[
\sum_{k=0}^{K-1} \Pi (x^{k+1},x ) \le 
\frac{1}{2} \sum_{i=1}^{n-1}  \big\| 
  z_{i}^{0} - z_{i}  \big\|_{M_i  }^2 .
\]
Then (ii) follows from the convexity of $\Pi$ with respect to $x^k$.
\end{proof}

\end{document}